\documentclass[a4paper,12pt,reqno]{amsart}
\usepackage[utf8]{inputenc}
\usepackage[T1]{fontenc}
\usepackage{lmodern}
\usepackage[english]{babel}
\usepackage{microtype}
\usepackage{comment}
\usepackage{amsmath,amssymb,amsfonts,amsthm,esint}
\usepackage{mathtools,accents}
\usepackage{mathrsfs}
\usepackage{aliascnt}
\usepackage{braket}
\usepackage{bm}

\usepackage[a4paper,margin=3cm]{geometry}
\usepackage[citecolor=blue,colorlinks]{hyperref}

\usepackage{enumerate}
\usepackage{xcolor}
\usepackage{graphicx}
\usepackage{mathbbol}

\makeatletter
\g@addto@macro\@floatboxreset\centering
\makeatother

\makeatletter
\def\newaliasedtheorem#1[#2]#3{
  \newaliascnt{#1@alt}{#2}
  \newtheorem{#1}[#1@alt]{#3}
  \expandafter\newcommand\csname #1@altname\endcsname{#3}
}
\makeatother

\numberwithin{equation}{section}

\newtheoremstyle{slanted}{\topsep}{\topsep}{\slshape}{}{\bfseries}{.}{.5em}{}

\theoremstyle{plain}
\newtheorem{theorem}{Theorem}[section]
\newaliasedtheorem{proposition}[theorem]{Proposition}
\newaliasedtheorem{question}[theorem]{Question}
\newaliasedtheorem{lemma}[theorem]{Lemma}
\newaliasedtheorem{corollary}[theorem]{Corollary}
\newaliasedtheorem{counterexample}[theorem]{Counterexample}

\theoremstyle{definition}
\newaliasedtheorem{definition}[theorem]{Definition}
\newaliasedtheorem{example}[theorem]{Example}
\newaliasedtheorem{conjecture}[theorem]{Conjecture}

\theoremstyle{remark}
\newaliasedtheorem{remark}[theorem]{Remark}
\newcommand{\setN}{\mathbb{N}}

\newcommand{\setR}{\mathbb{R}}

\newcommand{\eps}{\varepsilon}

\let\phi\varphi

\newcommand{\abs}[1]{\left\lvert#1\right\rvert}
\newcommand{\norm}[1]{\left\lVert#1\right\rVert}

\DeclareMathOperator{\Hess}{Hess}

\newcommand{\di}{\mathop{}\!\mathrm{d}}

\newcommand{\bs}{{\rm bs}}
\newcommand{\loc}{{\rm loc}}

\DeclareMathOperator{\Lip}{Lip}
\DeclareMathOperator{\Lipb}{Lip_b}
\DeclareMathOperator{\Lipbs}{Lip_\bs}
\DeclareMathOperator{\Liploc}{Lip_\loc}
\DeclareMathOperator{\Lipl}{Lip_{\mathrm{LOC}}}

\DeclareMathOperator{\vol}{vol}

\newcommand{\haus}{\mathscr{H}}
\newcommand{\Leb}{\mathscr{L}}

\newcommand{\dist}{\mathsf{d}}

\newcommand{\meas}{\mathfrak{m}}

\DeclareMathOperator{\CD}{CD}
\DeclareMathOperator{\RCD}{RCD}

\newfont{\tmpf}{cmsy10 scaled 2500}

\DeclareMathOperator{\Tan}{Tan}

\newcommand{\rstr}{\:\mbox{\rule{0.1ex}{1.2ex}\rule{1.1ex}{0.1ex}}\:}
\def\set{{\rm set}}
\newcommand{\mass}{{\mathbb M}}
\newcommand{\spt}{{\rm spt}}

\newcommand{\curr}[1]{[\![ #1 ]\!]}
\newcommand{\bM}{\operatorname{\mathbf M}}
\newcommand{\bN}{\operatorname{\mathbf N}}
\newcommand{\bI}{\operatorname{\mathbf I}}

\newcommand{\LWc}[2]{{\mathbf M}_{\text{\rm loc},\,#2}(#1)}  
\newcommand{\LWnc}[2]{{\mathbf N}_{\text{\rm loc},\,#2}(#1)} 
\newcommand{\LWirc}[2]{{\mathcal I}_{\text{\rm loc},\,#2}(#1)}
\newcommand{\LWic}[2]{{\mathbf I}_{\text{\rm loc},\,#2}(#1)}

\begin{document}
\author[A. Mondino]{Andrea Mondino}
\thanks{A. Mondino was funded by the European Research  Council
(ERC), under the European Union Horizon 2020 research and innovation programme, via the ERC Starting Grant “CURVATURE”, grant agreement No. 802689.}
\address[A. Mondino]{Mathematical Institute, University of Oxford, United Kingdom}
\email{andrea.mondino@maths.ox.ac.uk}

\author[R. Perales]{Raquel Perales}
\thanks{R. Perales was partially funded by 
the Austrian Science Fund (FWF) [Grant DOI: 10.55776/EFP6]}
\address[R. Perales]{Centro de Investigaci\'on en Matem\'aticas, 
De Jalisco s/n, Valenciana, Guanajuato, Gto. Mexico. 36023}
\email{raquel.perales@cimat.mx}

\title[GH and IF convergence of RCD and Kato spaces]{Gromov-Hausdorff and intrinsic flat convergence of $\RCD(K,N)$ and Kato spaces}

\begin{abstract} \emergencystretch=1em
We consider metric measure spaces $(X,\mathsf{d},\mathscr{H}^N)$ satisfying the curvature-free properties (ETR), (LDB), and whose regular set is almost everywhere quasi-convex. In particular, these assumptions are satisfied by non-collapsed $\RCD(K,N)$ spaces without boundary, as well as by non-collapsed strong Kato limit spaces without boundary. In this framework, we establish a local rigidity result for top-dimensional integral currents and deduce uniqueness, up to a sign, of orientations in the sense of metric currents. We then prove stability of orientation under pointed Gromov--Hausdorff convergence and show that, for both the $\RCD$ and strong Kato classes, the pointed Gromov--Hausdorff limit coincides with the local intrinsic flat limit. 
  The proof of the convergence is based on the local degree of (ETR) charts and on the $L^1$-stability of the degree under flat convergence.  A central ingredient is a full-measure biLipschitz decomposition of splitting charts around almost-Euclidean regular points of non-collapsed $\RCD(K,N)$ spaces.
\end{abstract}

\maketitle
{
\hypersetup{linkcolor=blue}
\setcounter{tocdepth}{2}
\tableofcontents
}

\section{Introduction}

Orientability of Ricci limit spaces was first investigated by Honda \cite{Honda17}, who studied it via the existence of a top-dimensional differential form with Sobolev regularity. More recently, Brena–Bru\'e–Pigati \cite{BrenaBruePigati2024} carried out a comprehensive study of orientability for non-collapsed 
$\RCD(K,N)$ spaces without boundary \cite{DPG}, providing several equivalent characterizations. These include: (i) Honda’s formulation in terms of Sobolev top-degree differential forms; (ii) topological orientability of the topological manifold part, whose complement is known to have Hausdorff codimension 2, \cite{CC:97, KapMonGT}; and (iii) the existence of a non-zero metric current with bounded density and zero boundary. They also proved stability of orientability within the non-collapsed $\RCD$ class, as well as stability of non-orientability for Ricci limit spaces. Their arguments primarily rely on the second characterization, namely topological orientability of the manifold part. As a notable application, they showed that any open 4-manifold with Euclidean volume growth and non-negative Ricci curvature must be orientable. Let us also mention \cite{JPRSS} for a study of orientability in the setting of Alexandrov spaces.
\smallskip

In this paper, we consider metric measure spaces $(X,\dist,\haus^N)$ belonging to one of the following two classes:
\begin{enumerate}
\item non-collapsed $\RCD(K,N)$ spaces without boundary;
\item non-collapsed strong Kato limit spaces without boundary.
\end{enumerate}
For both classes, we study orientability in the sense of metric currents, establish stability of orientability under pointed Gromov–Hausdorff convergence, and deduce that the pointed Gromov–Hausdorff and local flat limits coincide. The last assertion has been established by Sormani-Wenger \cite{SW1} in the framework of compact non-collapsed limits of smooth manifolds with no boundary and nonnegative Ricci curvature, 
and then extended to any uniform lower bound on the Ricci curvature by
Matveev-Portegies \cite{MatveevPortegies17}.
\smallskip

Compared with the earlier works \cite{SW1,MatveevPortegies17,Honda17,BrenaBruePigati2024}, this paper makes two new contributions. First, we give a direct argument using currents to prove uniqueness, up to a sign, of orientations, without relying on the equivalent characterizations of orientability established in \cite{BrenaBruePigati2024}. The argument is formulated under curvature-free assumptions and is based on local constancy together with an almost-everywhere quasi-convexity property of the regular set.  
Second, we extend known results on the convergence of currents under lower Ricci curvature bounds to Kato spaces.

\smallskip

 Metric currents were first introduced in the pioneering work of Ambrosio--Kirchheim \cite{AmbrosioKirchheim00} and further studied by Lang \cite{Lang} and Lang--Wenger \cite{LangWenger}. 
These are generalizations of currents in Euclidean space and smooth manifolds (see, e.g., \cite{Federer}). 

The starting point for this paper is the observation that, given an 
$N$-dimensional orientable Riemannian manifold 
$(M,g)$ endowed with a fixed orientation, one can canonically associate an $N$-dimensional integral current $T$ on $M$ whose mass measure  $\|T\|$ coincides with the Riemannian volume measure ${\rm dvol}_g=\haus^N$. Moreover, thanks to Stokes' theorem, $\partial M=0$ if and only if $\partial T= 0$.   The following definition is motivated by this observation; we refer the reader to Section \ref{SS:AKMetCur} for the notation and terminology.

\begin{definition}\label{def:Orientable} 
Let $(X,\dist,\haus^N)$ be a complete and separable metric measure space. We say that $(X,\dist,\haus^N)$  \textit{is orientable and with no boundary in the sense of currents} if there is a locally integral current $T\in \LWic{X}{N}$ such that $X=\set (T)$,  $\norm{T}=\haus^N$ and $\partial T=0$. 
\end{definition}

The next proposition is a local rigidity result for top-dimensional currents.  Its proof is independent of the uniqueness theorem of \cite{BrenaBruePigati2024} and will be given in Section~\ref{SubSec:Pf1213}.

\begin{proposition}\label{prop:locrep}
Let $(X,\dist,\haus^N)$ be an $\RCD(K,N)$ space without boundary for some $K\in\setR$ and $N\in\setN$, that is oriented and with no boundary in the sense of currents by $T\in\LWic{X}{N}$.  Then, for any $S\in\LWic{X}{N}$, $x\in X$ and $s>0$ such that
\[
 \|\partial S\|(B_s(x))=0,
\]
there exists $k\in\mathbb Z$ such that
\[
 S\rstr B_s(x)=k\,T\rstr B_s(x).
\]
\end{proposition}

From the previous proposition we obtain uniqueness of the current orientation up to a sign.

\begin{corollary}\label{cor-2orientations}
Let $(X,\dist,\haus^N)$ be an $\RCD(K,N)$ space without boundary, for some $K\in\setR$ and $N\in\setN$, that is oriented and with no boundary in the sense of currents by $T\in\LWic{X}{N}$.  Then the only two integral currents without boundary orienting $(X,\dist,\haus^N)$ are $T$ and $-T$.
\end{corollary}

For simplicity of presentation, the two results above are stated for $\RCD$ spaces. However, in the paper we prove Proposition~\ref{prop:locrep} and Corollary~\ref{cor-2orientations} in the more general setting of metric measure spaces satisfying the structural, but curvature-free, \emph{essential topological regularity} (ETR) assumption, and whose regular set is a.e.\ quasi-convex, see Proposition~\ref{prop:locrep1} and 
Proposition~\ref{prop:locrep2} for the precise statements. 

 Roughly speaking, (ETR) provides suitable topological charts around regular points together with a full-measure biLipschitz decomposition, see Definition \ref{def:assumptions}. The a.e.\ quasi-convexity assumption, see Definition \ref{def-ae-quasiconvex}, gives the quantitative connectivity of the regular set needed to localize the constancy argument. In Proposition~\ref{prop-ex-assms}, we prove that these assumptions are satisfied by non-collapsed $\RCD$ spaces and by non-collapsed strong Kato limit spaces, with the a.e.\ quasi-convexity property holding in the absence of boundary.
\smallskip

We briefly recall the notion of a Kato limit. 
Non-collapsed strong Kato limit spaces, as well as variants under weaker assumptions,  have been thoroughly  studied  by Carron--Mondello--Tewodrose \cite{CMT:GT, CarronMondelloTewodrose2023, CMT:AFM}.
In what follows all the manifolds are assumed to have empty boundary. For a complete Riemannian manifold $(M,g)$ of dimension $N\geq 2$, define
\begin{equation*}
 \mathrm{k}_t(M,g)= \sup_{x \in M}\int_0^t\int_M H(s,x,y)\textrm{Ric}_{-}(y) \textrm{dvol}_g(y) \di s,
\end{equation*}
where $H$ is the heat kernel of $M$ and $\textrm{Ric}_{-} : M \to \setR_+$
is the negative part of the smallest eigenvalue of the Ricci tensor. Following \cite{CarronMondelloTewodrose2023}, 
we say that a family $\Lambda$ of complete $N$-dimensional Riemannian manifolds satisfies a \emph{strong Kato bound} if there exists $t_0>0$ and a monotone non-decreasing function $f:(0,t_0]\to \setR_+$ with 
\begin{equation}\label{eq:f}
\int_0^{t_0}\frac{f(t)}{t}\,\di t<\infty,
\qquad
f(t_0)\leq\frac1{N-2},
\end{equation}
such that each $(M,g)\in \Lambda$ satisfies
\begin{equation}\label{eq:Kato}
\mathrm{k}_t(M,g)\leq f(t), \qquad \forall t \in (0,t_0]. 
\end{equation}
The upper bound in \eqref{eq:f} is understood as $+\infty$ when $N=2$.
Kato bounds are implied by uniform lower bounds on the Ricci curvature or by  suitable uniform $L^p$ estimates on $\textrm{Ric}_{-}$.

\begin{definition}[Carron--Mondello--Tewodrose]\label{def:strongKatoLimit}
A metric measure space $(X,\dist, \haus^N)$ is a \emph{non-collapsed strong Kato limit space} if there exists a sequence $\Lambda=\{(M_i, g_i )\}_{i\in \setN}$ of complete $N$-dimensional Riemannian manifolds satisfying
\begin{itemize}
\item a strong Kato bound (in the aforementioned sense); 
\item volume non-collapsing, i.e, there exists $v>0$ such that ${\rm vol}_{g_i}(B_{\sqrt{t_0}}(x_i))>v$, where $x_i\in M_i$ is a marked point; 
\end{itemize}
such that $(M_i, g_i, x_i)$ converges to $(X,\dist, x)$ in the pointed Gromov-Hausdorff sense.
\end{definition}

Next, we compare the pGH limits and local flat limits of sequences of RCD spaces and strong Kato limits oriented in the sense of currents.  For details about the local flat topology see  Definition \ref{def:locFlatTop} and Theorem \ref{thm-compactnessLW}.

\begin{theorem}\label{thm-main}
Let $(X_i,\dist_i,\haus^N,x_i)$ be  a sequence of  $\RCD(K,N)$ spaces,  for some $K \in \setR$ and $N \in \mathbb N$, converging in the pGH topology to $(X,\dist,x)$. 
Assume that each $X_i$ has no boundary and is oriented by the $N$-dimensional local integral current $T_i \in \LWic{X_i}{N}, \ \partial T_i=0$. Then, either one of the following holds:
\begin{itemize} 
\item In the collapsing case, where
$\lim_{i \to \infty}\haus^N(B_1(x_i)) =0$, 
the sequence $T_i$  converges in the local flat topology to the zero current $T=0 \in \LWic{X}{N}$.
\item In the non-collapsing case, where
$\lim_{i \to \infty}\haus^N(B_1(x_i)) >0$, it holds that
the currents $T_i$ subconverge in the local flat topology to a non zero current $T \in \LWic{X}{N}$ with $\partial T=0$, $\|T\|=\haus^N$, $X=\set(T)$. In particular, $(X,\dist,\haus^N)$ is orientable and has no boundary in the sense of currents.
\end{itemize}
\end{theorem}

 An analogous result holds for non-collapsed strong Kato limits: 

\begin{theorem}\label{thm-main2}
Let $(M_i,g_i,x_i)$ be a non-collapsing sequence of complete oriented Riemannian manifolds of dimension $N$, satisfying a strong Kato bound, converging in the pGH topology to $(X,\dist,x)$. Then $(X,\dist,\haus^N)$ is orientable and has no boundary in the sense of currents.

More precisely, denoting by  $T_i \in \LWic{M_i}{N}, \ \partial T_i=0, \ \|T_i\|={\rm dvol}_{g_i}$, a local integral current orienting $M_i$, it holds that $T_i$  subconverge in the local flat topology to a non zero current $T \in \LWic{X}{N}$, with $\partial T=0$, $\|T\|=\haus^N$, and $X=\set(T)$.
\end{theorem}

\begin{remark}
 The statements in Theorems \ref{thm-main} and \ref{thm-main2}, asserting that non-collapsed limits of spaces without boundary still have no boundary (in the sense of currents), should be compared with the analogous property — proved by Cheeger–Colding \cite{CC:97} and later extended to non-collapsed $\RCD$ spaces \cite[Thm.\ 5.1]{KapMonGT} — that non-collapsed Ricci limits of manifolds without boundary also have no boundary (in the sense of tangent cones).
\end{remark}

\begin{remark}\label{rmrk-GHIFcase}
Under the hypotheses of 
Theorem \ref{thm-main} (resp.\ Theorem \ref{thm-main2}), if additionally the metric spaces $(X_i,\dist_i)$ (resp.\ $(M_i,g_i)$) have uniformly bounded diameters, then it is not necessary to fix points $x_i$, and pGH convergence can be replaced with GH convergence. Furthermore,
following the terminology of integral current spaces introduced by Sormani and Wenger \cite{SW2}, one can add that
\begin{itemize} 
 \item In the collapsing case, the intrinsic flat limit
is the zero integral current space
$(\emptyset,0,0)$, although the Gromov--Hausdorff
limit $X$ may be nonempty.
\item In the non-collapsing case,  $(X,\dist,T)$ is an $N$-dimensional integral current space. Moreover, the Gromov-Hausdorff and intrinsic flat limits agree. 
\end{itemize}

Concerning the second statement, recall that the question of the equivalence GH$=$IF was originally motivated by Wenger’s compactness theorem \cite{Wenger} and by the introduction of intrinsic flat (IF) convergence by Sormani–Wenger \cite{SW2}. This problem has since been addressed in several works under a variety of assumptions; see, for instance, \cite{SW1, MatveevPortegies17, LiPerales, PeralesBdry, PorSor, PeralesNZ}.
\end{remark}

\subsection*{Organization of the paper and ideas of the proofs}

Section~\ref{SS:AKMetCur} recalls the facts about metric currents, degree of a Lipschitz map with respect to a current, intrinsic flat convergence, and local currents that we use. In Section~\ref{SS:CritHarm} we prove a countable full-measure biLipschitz decomposition for $\delta$-splitting maps in non-collapsed $\RCD(K,N)$ spaces, which is a key technical ingredient of the paper.  The proof uses the transformation theorem (see \cite[Proposition 3.13]{BrueNaberSemola20} for a proof in the $\RCD$ setting, after \cite{CJN:21}) combined with a density/maximal-family argument (cf.\ \cite[Theorem 6.5]{MondinoNaber19}).
In Section~\ref{Sec:RegPropmms}, we introduce the curvature-free properties (ETR) and (LDB), together with the a.e.\ quasi-convexity of the regular set, and we show that these assumptions are satisfied by non-collapsed $\RCD$ spaces and by non-collapsed strong Kato limit spaces. The lower density bound (LDB) will also be used to pass from full $\haus^N$-measure of the canonical set of a current to equality of this set with the whole ambient space.  In Section~\ref{sec-localdegree} we combine the (ETR) property with the Euclidean constancy theorem for local currents to show that 
the local degree of (ETR) charts with respect to
an orienting current is $+1$ or $-1$.

In Section~\ref{SubSec:Pf1213} we prove Proposition~\ref{prop:locrep} and Corollary~\ref{cor-2orientations}.  The quasi-convexity estimate ensures that the curves used to propagate local constancy remain inside the ball where the boundary vanishes.  In Section~\ref{sec-ProofsThms} we combine the local degree statement with the $L^1$-estimate of Matveev--Portegies \cite{MatveevPortegies17} for the degree.  On every regular ball, the degrees of the approximating currents take values in $\{1,-1\}$; their $L^1$ convergence therefore implies that the limiting degree also has absolute value one.  The pushforward mass inequality then yields the sharp lower density estimate for the limiting current.  The opposite inequality follows from volume convergence and lower semicontinuity of mass.  Hence the mass density of the limit is one at almost every regular point, which gives $\|T\|=\haus^N$ and $\set(T)=X$.  The non-compact proof is obtained by the same local argument after applying Lang--Wenger compactness. 

The proof for strong Kato limits uses the corresponding regular-ball estimates of Carron--Mondello--Tewodrose \cite{CarronMondelloTewodrose2023} and the same current-theoretic argument.

\bigskip

{\bf Acknowledgements:} The authors gratefully acknowledge support from the Simons Center for Geometry and Physics where they had the opportunity to meet as part of the Program {\em Geometry and Convergence in Mathematical General Relativity} in September 2025. The authors wish to thank J.\ Portegies and D.\ Semola for helpful discussions, and the anonymous referee for the careful reading of the manuscript and for many insightful comments  that significantly improved the exposition.

\section{Preliminaries}

\subsection{RCD spaces}
Throughout the paper, $(X,\dist)$ is a complete and separable metric space, endowed with a Borel non-negative measure $\meas$.  The triple $(X,\dist, \meas)$ is called a \emph{metric measure space}. When a reference point $\bar{x}$ is fixed, the tuple $(X,\dist, \meas, \bar{x})$ is called \emph{a pointed metric measure space}.

On metric measure spaces $(X,\dist, \meas)$, Sturm \cite{Sturm:I, Sturm:II} and Lott-Villani \cite{LV:AnnMath} introduced the $\CD(K,N)$ condition, encoding that the Ricci curvature is bounded below by $K\in \mathbb{R}$ and the dimension is bounded above by $N\in [1,\infty]$ in a synthetic sense, via optimal transport. In order to isolate the Riemannian-like structures from the Finslerian ones, Ambrosio-Gigli-Savar\'e \cite{AGS:14b} and Gigli \cite{G:15} (see also \cite{AGMR}), reinforced the Lott-Sturm-Villani $\CD$ condition by requiring linearity of the heat flow or, equivalently, that the Sobolev space $W^{1,2}(X,\dist, \meas)$ is a Hilbert space. The resulting "Riemannian-like" curvature-dimension condition is known as $\RCD(K,N)$. The $\RCD(K,N)$ condition can be equivalently characterized in terms of the Bochner inequality \cite{AGS:15, EKS:15, AMS:19} and it satisfies a local-to-global property \cite{CaMi:21} (see also \cite{Li:24}). Inspired by Cheeger-Colding theory of Ricci-limit spaces \cite{CC:97}, an $\RCD(K,N)$ is said \cite{DPG}  to be \emph{non-collapsed} if $\meas=\haus^N$, the $N$-dimensional Hausdorff measure (see also \cite{H:20, BGHZ:23}). In this case, $N$ must be an integer. For more about $\CD$ and $\RCD$ spaces we refer to the surveys \cite{Vil:BS, Amb:ICM, Sturm:ECM, G:Survey}.

\subsubsection{Singular strata and boundaries}

We say that a pointed metric space \ $(Y,\dist_Y, y)$ is \emph{tangent} to $(X,\dist)$ at $x \in X$ if there exists a sequence $r_i\downarrow 0$ such that 
 \[
 (X,r_i^{-1}\dist,x)\rightarrow(Y,\dist_Y,y)
 \]
 in the pGH-topology. The collection of all the tangent spaces of $(X,\dist)$ at $x$ is denoted by $\Tan_x(X,\dist)$. A compactness argument originally due to Gromov \cite{Gromov:Book}, yields that $\Tan_x(X,\dist)$ is non-empty for every $x\in X$, if $(X, \dist, \haus^N)$ is a non-collapsed $\RCD(K,N)$ space.  The set of \emph{regular} points is defined as 
\begin{equation}\label{eq-regularset}
\mathcal{R}:=\left\lbrace x\in X: \Tan_x(X,\dist)=\{(\setR^N,\dist_{\mathrm{eucl}})\}\right\rbrace 
\end{equation}	
and, for any $0\le k\le N-1$, let
\begin{equation}
\mathcal{S}^k:=\left\lbrace x\in X: \text{no tangent space at $x$ splits off $\setR^{k+1}$}\right\rbrace\,. 
\end{equation} 
The following   \textit{stratification of the singular set} $\mathcal{S}:=X\setminus\mathcal{R}$ holds: 
\begin{equation}
\mathcal{S}^0\subset \mathcal{S}^1\subset\dots\subset\mathcal{S}^{N-1}=\mathcal{S}.
\end{equation}

A quantitative analysis of the regular and singular strata leads to the definition of a $(k,\delta)$-symmetric ball and a $\delta$-splitting map \cite{ChNa:13}. A ball $B_s(x) \subset X$
is said to be $(k,\delta)$-symmetric provided 
\[
\dist_{GH}(B_s(x),B_s(z)) \le \delta s,
\]
where $z\in  Z \times \setR^k$ is a tip of the metric cone $C(Z) \times \setR^k$, for some metric space $(Z,\dist_Z)$.  The notion of $\delta$-splitting map recalled below has been extremely powerful in developing a  structure theory of Ricci limits \cite{CC:97, CC:II, CC:III} and $\RCD(K,N)$ spaces \cite{MondinoNaber19, BruePasqualettoSemola21}.

\begin{definition}\label{def:deltasplitting}
Let $N\in [1,\infty)$ and $k\in \setN, k\leq N$.
	Let $(X,\dist,\meas)$ be an $\RCD(-(N-1),N)$ space, $x\in X$ and $\delta>0$ be fixed.
	We say that $u=(u_1,\ldots,u_k):B_r(x)\to \setR^k$ is a $\delta$-splitting map provided 
	\begin{itemize}
		\item[(i)]
  $u_i$ is harmonic and $|\nabla u_i|< C(N)$, for all $i=1,\ldots,k$;
		\item[(ii)] $r^2\fint_{B_r(x)} |\Hess u_i|^2\di \meas <\delta$, for all $i=1,\ldots,k$;
		\item[(iii)] $\fint_{B_r(x)} |\nabla u_i\cdot \nabla u_j-\delta_{ij}|\di \meas <\delta$, for all $i,j=1,\ldots,k$.
	\end{itemize}
\end{definition}

For the proof of the next theorem, we refer to \cite[Theorem 3.8]{BrueNaberSemola20}; see also \cite{BruePasqualettoSemola21} for similar results.

\begin{theorem}[$\delta$-splitting vs $\eps$-GH isometry]
	\label{splitting vs isometry}
	Let $1\le N<\infty$ be fixed. 
	
	\begin{itemize}
		\item[(i)] For every $0<\delta<1/2$ and $\eps\le \eps(N,\delta)$, the following holds. If $(X,\dist,\meas)$ is an $\RCD(-\eps(N-1),N)$ space  satisfying
		\begin{equation}\label{eq:zzz5}
		\dist_{mGH}(B_2(x),B_2^{\setR^k\times Z}(0,z))\le \eps
		\end{equation}
		for some integer $k$, some $x\in X$ and some pointed metric space $(Z,\dist_Z, z)$, then
		there exists a $\delta$-splitting map \[
  u=(u_1,\ldots,u_k):B_1(x)\to \setR^k.
  \]
		
		\item[(ii)] For every $\eps>0$ and $\delta < \delta(N,\eps)$ the following holds. If  $(X,\dist,\meas)$ is a normalized $\RCD(-\delta(N-1),N)$ space and there exists a $\delta$-splitting map $u:B_6(x)\to \setR^k$ for a given $x\in X$, then
	    \begin{equation}
	     \dist_{GH}(B_{1/k}(x),B^{\setR^k\times Z}_{1/k}(0,z))< \eps
	    \end{equation}
    	for some pointed metric space $(Z,\dist_Z,z)$. Moreover, there exists $f:B_{1}(x)\to Z$ such that
    	\begin{equation}
    	(u-u(x),f):B_{1/k}(x)\to B_{1/k}^{\setR^k\times Z}(0,z)
    	\quad \text{is an $\eps$-GH isometry.}
	    \end{equation} 	 
	    \item[(iii)] If we additionally assume that $(X,\dist,\haus^N)$ is a $\RCD(-\delta(N-1),N)$ non-collapsed space with $\haus^N(B_1(x))>v>0$, $k=N-1$, and $\delta<\delta(N,v,\eps)$, then $(Z,\dist_Z,z)$ in (ii) can be chosen to be the ball of a one dimensional Riemannian manifold, possibly with boundary.
	
    \end{itemize}
\end{theorem}

Two different notions of boundary for a non-collapsed $\RCD(K,N)$ space $(X,\dist,\haus^N)$ have been proposed \cite{DPG,KapMonGT}; however, it follows from \cite[Theorem 6.6.(i)]{BrueNaberSemola20} that the fact of $X$ having empty boundary is equivalent in the two formulations.  Throughout this note, we will assume that $X$ has empty boundary, i.e., $\mathcal{S}^{N-1}\setminus\mathcal{S}^{N-2}=\emptyset$. By  \cite{DPG}, this is equivalent to requiring that
\begin{equation}\label{eq:SingCod>1}
{\rm Codim_{Hauss}}(X\setminus \mathcal R)>1,
\end{equation} 
where ${\rm Codim_{Haus}}$ denotes the Hausdorff codimension.

The following theorem gives topological regularity for a neighbourhood of a regular point and  useful regularity properties of  $\delta$-splitting maps. In the $\RCD$ setting, the former was proved in \cite{KapMonGT} employing Cheeger-Colding's  Reifenberg thorem for metric spaces \cite{CC:97}, while the latter  was proved in \cite{BrueNaberSemola20}. Both were previously established for Ricci limits, respectively by Cheeger-Colding \cite{CC:97} and by Cheeger-Jiang-Naber \cite{CJN:21}.

\begin{theorem}\label{thm:biholder}
	Let $N\in \setN,\ N\geq 1.$  For each $\eps \in (0, 1/5)$ there exists $\delta=\delta(\eps,N)>0$ such that for any $\RCD(-\delta(N-1),N)$ space $(X,\dist,\haus^N)$ and for any $(N, \delta)$-symmetric ball $B_{16}(p)\subset X$, the following holds: 
 \begin{itemize}
 \item The ball $B_8(p)$ is homeomorphic to a smooth $N$-dimensional manifold without boundary.
 \item There exists a map
$u:B_8(p)\to\setR^N$, with $u(p)=0$,
verifying the following properties:
	\begin{itemize}
		\item[i)] $u:B_8(p)\to\setR^{N}$ is an $\eps$-splitting map;
		\item[ii)] there exists a closed set $U\subset B_1(p)$ such that 
		\begin{equation}\label{eq:volest}
		\haus^{N}\left(B_1(p)\setminus U\right)\le \eps,
		\end{equation}
        and, for all $x,y\in U$:
		\begin{equation}\label{eq:bilipest}
		(1-\eps)\dist(x,y)\le \abs{u(x)-u(y)}\le (1+\eps)\dist(x,y); 
		\end{equation}
		\item[iii)] for all $x,y\in   B_1(p)$ it holds that \begin{equation}\label{eq:holderest}
		(1-\eps)\dist(x,y)^{1+\eps}\le \abs{u(x)-u(y)}\le (1+\eps)\dist(x,y)\, ;
		\end{equation}
		\item[iv)] $u(B_1(p))\supset B_{1-2\eps}^{\setR^{N}}(0)$;
		\item[v)] the following volume estimate holds: \begin{equation}\label{eq:sharpahlfors}
(1-\eps)\omega_{N}\le\haus^{N}(B_1(p))\le (1+\eps)\omega_{N}\, . 
\end{equation}
	\end{itemize}
 \end{itemize}
	\end{theorem}

In the body of the paper, the $\varepsilon>0$ appearing in the second item will be referred to as the \emph{splitting error}.

\subsection{Ambrosio-Kirchheim's metric currents} \label{SS:AKMetCur}
 Currents with finite mass in complete metric spaces were introduced by Ambrosio-Kirchheim \cite{AmbrosioKirchheim00}. 
 The theory was extended to currents with non-finite mass and in non necessarily complete spaces by Lang \cite{Lang} and Lang-Wenger \cite{LangWenger}. 
 Here we review, in a non-extensive way, some of  Ambrosio-Kirchheim's results that we will need. 

\subsubsection{Definition and basic properties}

\smallskip

Let $(X,\dist)$ be a complete metric space. 
We denote by $\Lip(X)$  the set of Lipschitz functions on $X$ and by
$\Lipb(X)$ the bounded ones.   For $n \geq 0$ an integer, let  
 \[
 \mathcal D^n(X)= \Lipb(X) \times [\Lip(X)]^{n}.
 \]
 Elements in 
 $\mathcal D^n(X)$ will usually be denoted as $(f, \pi)$, where the first entry is an element in  $\Lipb(X)$ and $\pi=(\pi_1,\dots, \pi_n) \in [\Lip(X)]^{n}$.  Sometimes $(f, \pi)$ will be denoted as 
 $f\,d\pi$ (to resemble an $n$-differential form).

An $n$-dimensional current $T$ on $X$ 
is a multilinear map $T: \mathcal D^n(X)  \to \mathbb R$  that satisfies:
\begin{itemize}
    \item $T$ is continuous: for any sequence $(f, \pi^i)  \in \mathcal D^n(X)$ such that $\pi^i$ converge pointwise to $\pi \in [\Lip(X)]^n$ and such that  $\sup_{i,j}\Lip(\pi^i_j) < \infty$, it holds 
    \[
    \lim_{i \to \infty}T(f, \pi^i)= T(f, \pi) ; 
    \]
    \item  $T$ is a local map:     
    $T(f, \pi)=0$, whenever there exists $j$ such that $\pi_j$ is constant on a neighborhood of $\{f \neq 0\}$;
    \item $T$ is bounded:
    there exists a 
    finite Borel measure $\mu$ on $X$ such that 
    \begin{equation}\label{eq-massMeas}
        |T(f, \pi)| \leq \Lip(\pi_1)\cdots \Lip(\pi_n) \int_X |f| \di \mu,
    \end{equation}
    for all $(f, \pi) \in \mathcal D^n(X)$.
\end{itemize}

 The vector space of all $n$-dimensional currents in $X$ is denoted by 
$\bM_n(X)$.

The smallest measure $\mu$ that 
satisfies \eqref{eq-massMeas}
is denoted by 
$||T||$. It is
called the mass measure of $T$. 
The mass and the support of $T$ are defined as $\mass(T)=||T||(X)$ and 
$\spt(T)= \spt(\norm{T})$.

When $(X,\dist)$ is the standard Euclidean space of dimension $n$, a prime example of an $n$-dimensional current is  
\[
[[\theta]]
(f, \pi)=  \int_{\setR^n} \theta(x) f(x)  \det( D_x\pi) \di\mathcal L^n(x),
\]
where  $\theta \in L^1(\setR^n, \mathbb R)$ and $D_x\pi$ denotes the derivative of $\pi$ at $x$, which exists $\mathcal L^n$-a.e. by Rademacher's theorem.

Given another complete metric space $Y$ and a Lipschitz function $\Psi: X \to Y$,
the \emph{push-forward} of $T$ under $\Psi$ is the current ${\Psi}_{\sharp} T: \mathcal D^n(Y)
\to \mathbb R$ defined as
\[
{\Psi}_{\sharp} T (f, \pi) 
= T( f\circ \Psi, \pi\circ\Psi), \qquad \forall(f, \pi) \in \mathcal D^n(Y).
\]
It holds that
\begin{equation}\label{eq-pushMass}
\|\Psi_\sharp T \| \leq \mathrm{Lip}(\Psi)^n\Psi_\sharp\|T\|.
\end{equation}

For $n \geq 1$, the  boundary of $T$ is the multilinear functional 
\[
\partial T: \Lipb(X) \times [\Lip(X)]^{n-1}  \to \mathbb R,
\]
defined as
\[
\partial T(f, \pi_1,\ldots,\pi_{n-1}) = T(1, f, \pi_1,\ldots,\pi_{n-1}),
\]
where $1:  X \to \mathbb R$  denotes the constant function equal to $1$.
We say that $T\in \bM_{n}(X)$ is a \emph{normal} current if $\partial T \in \bM_{n-1}(X)$. The vector space of all
normal $n$-currents on $X$ is denoted by $\bN_n(X)$. For $n=0$, we set $\bN_0(X)=\bM_0(X)$.

For any integer $0 \leq k \leq n$ and $(g, \tau) \in \mathcal D^k(X)$, the \emph{restriction} of 
$T$ to $(g, \tau)$ is the current
\[
T \rstr (g, \tau):  
\mathcal D^{n-k}(X) \to \mathbb R
\]
given by $(T\rstr(g,\tau))(f,\pi)=T(fg,\tau,\pi)$. It holds that
\begin{equation}\label{eq-mass-Restr}
\norm{T \rstr (g, \tau_1,\ldots,\tau_k)}  \leq   \sup|g| \Lip(\tau_1)\cdots \Lip(\tau_k)\norm{T}.
\end{equation}

Currents $T \in \bM_n(X)$ can be naturally extended to $\mathcal B^\infty(X) \times [\Lip(X)]^n$, where $\mathcal B^\infty(X)$ denotes the class of bounded Borel functions on $X$. This extension is denoted in the same way and satisfies natural product and chain rules for derivatives, as well as continuity, locality and boundedness properties, see \cite[Theorem 3.5]{AmbrosioKirchheim00}.  In particular, for a
Borel set $A \subset X$, the restriction of $T$ to $A$, is the current $ T \rstr A: \mathcal D^n(X)  \to \mathbb R$
given by 
\[
 T   \rstr A ( f, \pi) = T( 1_A \, f, \pi),  \quad \forall (f, \pi) \in \mathcal D^n(X),
\]
where $1_A:  X \to \mathbb R$  denotes the indicator function of $A$. In this case $||T   \rstr A|| = ||T ||  \rstr A$.

\subsubsection{Integral currents}

Let $n \geq 1$ be an integer, $ T \in \bM_n(X)$ is called \emph{rectifiable} if: 
\begin{itemize}
    \item $||T||$ is concentrated on an $\mathcal H^n$-rectifiable subset of $X$;
    \item $||T||$ vanishes on  $\mathcal H^n$-negligible sets. 
\end{itemize}
If additionally, for all Lipschitz maps $\Psi: X \to \mathbb R^n$ and all open sets $A \subset X$,
it holds 
\[
\Psi_{\sharp} ( T  \rstr A ) = 
\curr{\theta}  \quad \text{ for some } \quad \theta \in L^1(\mathbb R^n, \mathbb Z),
\]
we say that $T$ is \emph{integer rectifiable}. 
The collection of $n$-rectifiable currents is denoted by $\mathcal R_n(X)$. 
 The collection of $n$-integer rectifiable currents is denoted by $\mathcal I_n(X)$. Furthermore, the class of $n$-dimensional \emph{integral currents} is defined as 
\[
\bI_n(X) = \mathcal I_n(X) \cap \bN_n(X).
\]

The \emph{(canonical) set} of $T \in \mathcal R_n(X)$ defined by 
\begin{equation}\label{eq:defSet}
\set(T)=\Big \{ x \in X \,|\, \liminf_{r \to 0}\frac{ \norm{T}(B_r(x))}{\omega_n r^n} >0\Big \}
\end{equation}
is countably $\haus^n$-rectifiable and 
$\norm{T}$ is concentrated on $\set(T)$. Moreover, any Borel set on which $\norm{T}$ is concentrated must contain $\set(T)$,
up to an $\haus^n$-negligible set, see {\cite[Theorem 4.6]{AmbrosioKirchheim00}}.

We next recall a slight variant of the classical constancy theorem for currents in Euclidean spaces
{\cite[4.1.7]{Federer}}, whose corollary below will be crucial in the proof of the main results of this paper.

\begin{theorem}\label{thm:topCurrentsEuc}
Let $S \in \bN_n(\setR^n)$. Then   there exists a unique function $g \in \mathrm{BV}(\setR^n)$ that satisfies  
\[ 
S = [[ g ]].
\]
Moreover, 
\[
\norm{\partial S} = |Dg|,
\]
where $Dg$ is the derivative in the sense of distributions of $g$ and $|Dg|$ denotes its total 
variation. If $S \in \bI_n(\setR^n)$, then $g \in \mathrm{BV}(\setR^n)$  takes values in $\mathbb{Z}$.

If, in addition, $|| \partial S || (U)= 0$ for some connected open set $U \subset \mathbb R^n$, then $g$ is constant on  $U$.   
\end{theorem}

\begin{proof}
The representation and boundary-mass identity follow from
\cite[Theorem~3.7]{AmbrosioKirchheim00}.
If $S$ is integral, its multiplicity $g$ is integer-valued
almost everywhere.
Finally, $\|\partial S\|(U)=0$ implies $|Dg|(U)=0$,
so $g$ is constant almost everywhere on the connected
open set $U$.
\end{proof}

\subsubsection{Parametric representation of rectifiable currents}
In this section,  we define  \emph{charts, atlases and weights} for rectifiable currents. We represent a rectifiable current and its mass measure in terms of these objects and note the relation between the degree as in Definition \ref{def-degree} and weight of a current.

\begin{definition}\label{def-charts}
Let $T \in \bM_n(X)$. Let $f_i: K_i \to X$, $i\in \setN$, be 
a countable collection of maps that are biLipschitz to its images with the following properties:
\begin{itemize}
\item Each subset $K_i\subset \mathbb R^n$ is compact;
\item  $f_i(K_i)$, $i\in \setN$, are pairwise disjoint;
\item $\haus^n(\set(T)\setminus \bigcup_{i} f_i(K_i))=0$.
\end{itemize}
We say that each $(f_i,K_i)$ is a \emph{chart} for $\set(T)$ and the collection of charts is called an \emph{atlas} for $\set(T)$.
\end{definition}

We note that by  \cite[Lemma 4.1]{AmbrosioKirchheim00}, for any $T \in \mathcal R_n(X)$, there always exists an atlas for
$\set(T)$.

\begin{theorem}[{\cite[Theorem 4.5]{AmbrosioKirchheim00}}]\label{thm:parametricRep}
Let $T \in \mathcal R_n(X)$ and let $(f_i, K_i)_i$ be an atlas for $\set(T)$. 
Then there exist 
functions $\theta_i \in L^1(\setR^n)$  with $\spt(\theta_i) \subset K_i$  such that 
\begin{equation}\label{eq:Tsum}
T = \sum_{i}f_{i\sharp} [[\theta_i]]   \qquad \text{and} \qquad \mass(T)=\sum_i\mass(f_{i\sharp} [[\theta_i]]).
\end{equation}
Conversely, any current $T \in \bM_n(X)$ that can be written as in \eqref{eq:Tsum} is an element of $\mathcal R_n(X)$.   Additionally, $T\in \mathcal I_n(X)$ if and only if the $\theta_i$'s are integer valued. 
\end{theorem}

\begin{definition}\label{def:theta_T}
Let $T\in \mathcal R_n(X)$ and let $(f_i,K_i)_i$ be an atlas for $\operatorname{set}(T)$.  With the notation of
Theorem~\ref{thm:parametricRep}, we say that $\theta_i$ is the weight (or multiplicity) with respect to $f_i$.  We let 
$\theta_T :  \set(T)  \to \setR $ be the Borel function given by 
\[
\theta_T(x):=\theta_i(f_i^{-1}(x)),\qquad x\in f_i(K_i),
\]
and called it the weight (or multiplicity) of $T$ with respect to $(f_i,K_i)_i$. 
\end{definition}

The signed multiplicity $\theta_T$ depends on the chosen
parametrizations, whereas $|\theta_T|$ is independent of
them up to $\haus^n$-negligible sets, as explained below.

\begin{remark}[Change of parametrizations]
\label{rmrk-pushforwardweights}
Suppose that $(f_i,K_i,\theta_i)_i$ and
$(g_j,L_j,\vartheta_j)_j$ are two biLipschitz
parametric representations of the same rectifiable
current $T$, as in Theorem~\ref{thm:parametricRep}.
On an overlap, set
\[
 A_{ij}:=
f_i^{-1}\bigl(f_i(K_i)\cap g_j(L_j)\bigr),
 \qquad
h_{ij}:=g_j^{-1}\circ f_i:A_{ij}\to L_j.
\]
The transition map $h_{ij}$ is biLipschitz.
The change-of-variables formula gives
\[
\vartheta_j(h_{ij}(b))
 =
 \theta_i(b)\,
\operatorname{sgn}\det\bigl(\operatorname{ap}D h_{ij}(b)\bigr)
 \qquad\text{for $\mathcal L^n$-a.e. }b\in A_{ij}.
\]
Here $\operatorname{ap}D$ denotes the approximate derivative
on the measurable parameter set. Since $h_{ij}$ is
biLipschitz, its determinant is nonzero almost everywhere.
Consequently,
\[
 \bigl|\vartheta_j(g_j^{-1}(x))\bigr|
 =
 \bigl|\theta_i(f_i^{-1}(x))\bigr|
 \qquad\text{for $\haus^n$-a.e. }
 x\in f_i(K_i)\cap g_j(L_j).
\]
Thus changing parametrizations can change the sign
of the multiplicity, but not its absolute value.
\end{remark}

The mass measure of a rectifiable current $T\in\mathcal R_n(X)$ has the representation
\begin{equation}\label{eq-TmeasRep}
\ \|T\|=\lambda\,|\theta_T|\,\haus^n\rstr\set(T) ,
\end{equation}
where $\lambda:\set(T)\to[n^{-n/2},2^n/\omega_n]$ is the area factor of the approximate tangent normed space $\Tan^{(n)}(\set(T),x)$; see \cite[Theorem~9.5 and (9.11)]{AmbrosioKirchheim00}.

\subsection{Intrinsic Flat and Gromov-Hausdorff convergence}

Let $(X,\dist)$ be a complete metric space and  $T_1, T_2 \in \bI_{n}(X)$. The \emph{flat distance} between $T_1$ and $T_2$ is defined as 
\[
\dist_{F}^X( T_1, T_2)=\inf  \left\{  \mass(U)+ \mass(V) \, : \,  U \in \bI_{n}(X), \,\,\,  V   \in  \bI_{n+1} (X),\,\, T_2 -T_1=  U + \partial V       \right\}.
\]

If $T_i \to T$  in flat sense to  $T \in \bI_{n}(X)$, then
\begin{itemize}
\item  $\lim_{i\to\infty} \dist_{F}^X( \partial T_i , \partial T)=0$;
\item $\mass(T)\le \liminf_{i\to\infty}\mass(T_i)$;
\item If $\lim_{i \to \infty}\mass(T_i)=0$, then $T=0$. 
\end{itemize}

\begin{theorem}[{\cite[Theorem 1.2]{Wenger}}]\label{thm-WengerCpctness}
Let $(X_i,\dist_i)$ be a sequence of complete metric spaces and  $T_i \in \bI_{n}(X_i)$.   If 
 $$\sup_{i \in \mathbb N} \Big  \{ ||T_i || (X_i)  +   ||\partial T_i|| (X_i) \Big  \}  < \infty \quad  \text{and} \quad 
\sup_{i \in \mathbb N}\operatorname{diam}(\operatorname{spt}T_i)<\infty,
$$
 then there exist a complete metric space $W$, a subsequence $(i_j)_{j \in \mathbb N}$ and
isometric embeddings 
\[
\varphi_j: X_{i_j} \to W,  \qquad j \in \mathbb N,
\]  
such that 
 \[
 \dist_{F}^W(\varphi_{j \sharp} T_{i_j} , T) \to 0,
 \]
for some $T \in \bI_{n}(W)$.
\end{theorem}

For completeness, we prove the well-known fact that flat convergence also implies convergence of the currents restricted to balls, c.f. 
{\cite[Lemma 4.1]{Sormani2018}}. See Section \ref{SS:LWcur} below for the notion of locally integral currents.

\begin{lemma}\label{lem:localization-flat-balls}
Let $(W,\dist_W)$ be a complete metric space and let
\[
    S_i,S\in\bI_N(W)
\]
satisfy
\[
    \dist_F^W(S_i,S)\to 0.
\]
Let $z_i,z\in W$ be such that $z_i\to z$. Then, after passing to a
subsequence, there exists a set
\[
    \mathcal G\subset(0,\infty),
    \qquad
    \mathcal L^1((0,\infty)\setminus\mathcal G)=0,
\]
such that for every $R\in\mathcal G$ the currents
\[
    S_i\rstr B_R(z_i),
    \qquad
    S\rstr B_R(z)
\]
belong to $\bI_N(W)$ and
\begin{equation}\label{eq:localized-flat-convergence}
    \dist_F^W
    \bigl(
        S_i\rstr B_R(z_i),
        S\rstr B_R(z)
    \bigr)
    \to0.
\end{equation}

The same conclusion holds if
$S_i,S\in\LWic{W}{N}$ and $S_i\to S$ locally flat:
after passing to a single subsequence, there is a
full-measure set $\mathcal G\subset(0,\infty)$
such that the stated integrality and flat-convergence
conclusions hold for every $R\in\mathcal G$.
\end{lemma}

\begin{proof}
Passing to a subsequence, we may assume that
\[
    \dist_F^W(S_i,S)\leq 2^{-i}.
\]
Hence there exist
\[
    A_i\in\bI_N(W),
    \qquad
    B_i\in\bI_{N+1}(W),
\]
such that
\begin{equation}\label{eq:flat-decomposition-localization}
    S_i-S=A_i+\partial B_i
\end{equation}
and
\begin{equation}\label{eq:summable-flat-errors}
    \sum_{i=1}^{\infty}
    \bigl(
        \mass(A_i)+\mass(B_i)
    \bigr)
    <\infty.
\end{equation}
Set
\[
    \rho_i(w):=\dist_W(w,z_i),
    \qquad
    \rho(w):=\dist_W(w,z).
\]

By the slicing theorem \cite[Theorem 5.6]{AmbrosioKirchheim00}, after discarding a set of radii of
$\mathcal L^1$-measure zero, all the restrictions and slices appearing
below are integral currents. Moreover,
\[
    \int_0^\infty
    \mass\bigl(\langle B_i,\rho_i,R\rangle\bigr)\,dR
    \leq
    \mass(B_i).
\]
Together with \eqref{eq:summable-flat-errors} and Fubini's theorem,
this implies that for $\mathcal L^1$-a.e.\ $R>0$,
\begin{equation}\label{eq:slice-error-to-zero}
    \mass\bigl(\langle B_i,\rho_i,R\rangle\bigr)
    \to0.
\end{equation}
We also discard the negligible set of radii for which
$\|S\|(\{\rho=R\})>0$.

Fix a radius $R$ outside these exceptional sets. By the definition of
the slice,
\[
    \partial\bigl(B_i\rstr\{\rho_i<R\}\bigr)
    =
    (\partial B_i)\rstr\{\rho_i<R\}
    +
    \langle B_i,\rho_i,R\rangle.
\]
Restricting \eqref{eq:flat-decomposition-localization} to
$\{\rho_i<R\}$ therefore gives
\begin{align*}
    (S_i-S)\rstr B_R(z_i)
    &=
    A_i\rstr B_R(z_i)
    +
    \partial\bigl(B_i\rstr B_R(z_i)\bigr)\\
    &\qquad
    -
    \langle B_i,\rho_i,R\rangle.
\end{align*}
Consequently,
\begin{equation}\label{eq:ballconv}
    \dist_F^W
    \bigl(
        S_i\rstr B_R(z_i),
        S\rstr B_R(z_i)
    \bigr)
    \leq
\mass(A_i)+\mass(B_i)
    +
    \mass\bigl(\langle B_i,\rho_i,R\rangle\bigr)
     \to0.
\end{equation}

It remains to move the center in the restriction of the limit current.
Let
\[
    \eta_i:=\dist_W(z_i,z).
\]
Since $\eta_i\to0$,
\[
    B_R(z_i)\triangle B_R(z)
    \subset
    \{w\in W:|\rho(w)-R|\leq\eta_i\}.
\]
Therefore,
\begin{align*}
    \dist_F^W
    \bigl(
        S\rstr B_R(z_i),
        S\rstr B_R(z)
    \bigr)
    &\leq
    \mass\Bigl(
        S\rstr\bigl(B_R(z_i)\triangle B_R(z)\bigr)
    \Bigr)\\
    &\leq
    \|S\|
    \bigl(
        \{|\rho-R|\leq\eta_i\}
    \bigr).
\end{align*}
Since $\|S\|(\{\rho=R\})=0$, the last quantity tends to zero. Combining this with
\eqref{eq:ballconv} proves
\eqref{eq:localized-flat-convergence}.

For the locally flat version, fix
$0<a<b<\infty$.
For all sufficiently large $i$, the balls $B_b(z_i)$
are contained in the interior of
$C:=\overline B_{b+1}(z)$.
By Definition~\ref{def:locFlatTop}, there are locally
integral currents $A_i,B_i$ such that
\[
 S_i-S=A_i+\partial B_i,
 \qquad
 \|A_i\|(C)+\|B_i\|(C)\longrightarrow0.
\]
After passing to a subsequence, these local errors
are summable.
The preceding slicing argument on $R\in(a,b)$
uses only the masses on $C$ and therefore proves
the bounded-interval assertion.

Apply the bounded-interval argument successively on
$I_m=(1/m,m)$, for integer $m\geq2$, taking nested subsequences,
and then take a diagonal subsequence.
For each $m$, its tail is a subsequence of the one
chosen for $I_m$, so convergence holds outside a null
set $N_m\subset I_m$.
Since $\bigcup_{m\geq2}I_m=(0,\infty)$, convergence
holds for every radius outside $\bigcup_{m\geq2}N_m$.
Finally, slicing allows us to discard a further null
set so that all restrictions of $S$ and of the
countably many $S_i$ are integral.
This proves the assertion for almost every $R>0$.
\end{proof}

Sormani and Wenger introduced the notion of an integral current space and defined the intrinsic flat distance between them \cite{SW2}. An \emph{$n$-dimensional integral current space} $(X, \dist, T)$ consists of a metric space $(X, \dist)$ and an $n$-dimensional integral current, $T \in \bI_n(\overline{X})$,
where $\overline{X}$ denotes the metric completion of $(X,\dist)$, such that $\set(T)=X$.   The \emph{zero $n$-dimensional integral current space} denoted as ${\bf{0} }=(X,\dist,T)$, is defined by $T=0$ and $\set(T)=\emptyset$. We will use the notation $\mass(X, \dist, T)=\mass(T)$ and $\set((X,\dist,T))=\set(T)$. We say that $(X,\dist,T)$ is precompact if $(X,\dist)$ is precompact. The definition of intrinsic flat distance is as follows.

\begin{definition}[{\cite[Definition 1.1]{SW2}}]\label{IF-defn} 
Given two $n$-dimensional precompact integral current spaces, $(X_1, \dist_1, T_1)$ and $(X_2, \dist_2,T_2)$, the 
\emph{intrinsic flat 
distance} between them is defined as
\[
\dist_{\mathcal{F}}\left( (X_1, \dist_1, T_1), (X_2, \dist_2, T_2)\right)=\inf\left\{\dist_{F}^X(\varphi_{1\sharp}T_1, \varphi_{2\sharp}T_2):\,\,\varphi_j: X_j \to X \right\},
\]
where the infimum is taken over all  complete metric
spaces $X$ and all isometric embeddings  $\varphi_j$.
\end{definition}

Now we recall that if a sequence admits both a Gromov–Hausdorff limit and an intrinsic flat limit, then the two convergences can be realized within a single ambient metric space, and the intrinsic flat limit space is contained in the Gromov–Hausdorff limit space.

\begin{theorem}[Sormani-Wenger \cite{SW2}]\label{thm-IFsubsetGH}
Let $(X_i,\dist_i)$ be a sequence of compact metric spaces and  $T_i \in \bI_{n}(X_i)$.   If $(X_i, \mathsf d_i, T_i)$ are $n$-dimensional integral current spaces such that
 $$\sup_{i \in \mathbb N}  \Big  \{ ||  T_i || (X_i)  +   ||\partial T_i|| (X_i)  \Big  \}  < \infty,$$
and $(X_i,\dist_i)$ converges in the GH topology to $(X,\dist)$, then 
there exists a complete metric space $W$, 
a subsequence $(i_j)_{j \in \mathbb N}$,
isometric embeddings \[
\varphi_j: X_{i_j} \to W, \qquad j \in \mathbb N,
\]
$\varphi:X \to W$ and $T \in \bI_{n}(W)$ such that 
\[
 \dist_{F}^W(\varphi_{j \sharp} T_{i_j} , T) \to 0,
 \]
 
\[
    \dist_H^W
    \bigl(\varphi_j(X_{i_j}),\varphi(X)\bigr)
    \to 0,
\]
and
\[
    \set(T)\subset\spt(T)\subset\varphi(X).
\]
The last inclusion can be strict; in particular, cancellation
may occur and the limiting current may be zero.
\end{theorem}

\subsubsection{The degree of a Lipschitz map with respect to a current and the 
$L^1$-convergence of the degree}

Following the notation of Matveev-Portegies \cite{MatveevPortegies17}, 
we define the degree of a Lipschitz map with respect to a current. We also refer the reader to the classical monograph \cite[Ch.\ 4, Sect.\ 3]{GiaModSou}, for a related notion of degree in the framework of cartesian currents.

\begin{definition}\label{def-degree}
Let $X$ be a complete metric space,  $T \in \bI_n(X)$ and $\Psi:X \to \setR^n$ be a Lipschitz map.
The \emph{degree of $\Psi$ with respect to $T$}
is the unique function 
\[
\deg(T, \Psi , \cdot ) \in \mathrm{BV}(\setR^n)
\] 
taking values in $\mathbb Z$ such that   
\[
\Psi_\# T = [[ \deg(T,\Psi, \cdot)]] \in \bI_n(\setR^n).
\]
\end{definition}

The following estimate, due to Matveev--Portegies \cite{MatveevPortegies17},
controls the degree under simultaneous variation of
the current and the coordinate map.

 \begin{lemma}[Matveev-Portegies \cite{MatveevPortegies17}, c.f. Portegies-Sormani \cite{PorSor}]\label{le:L1DegConv}
Let $X$ be a complete metric space,  $T, S \in \bI_n(X)$ and  $\Phi,\Psi : X \to \setR^n$, be
  such that all components of  $\Phi=(\varphi_1,\ldots, \varphi_n)$ and $ \Psi=(\psi_1, \ldots, \psi_n)$ are $L$-Lipschitz functions. 
Then
\begin{align*}
\| \deg( T, \Phi, \cdot) - \deg(S, \Psi, \cdot)& \|_{L^1(\setR^n)}  \leq   \\
& L^n  \dist_{F}^X( T , S ) + 
2 L^{n-1}\sum_{j=1}^n \|\varphi_j - \psi_j\|_\infty  ( \mass(S) + \mass(\partial S) ).
\end{align*}
\end{lemma}

 We note that Lemma \ref{le:L1DegConv} is the key result to show that the weight functions of a converging sequence of currents converge to the weight function of the limit current (see Section~\ref{sec-ProofsThms}).

\subsection{Lang--Wenger's currents}\label{SS:LWcur}
The theory of Ambrosio-Kirchheim's currents was extended to currents with non-finite mass and in non necessarily complete metric spaces by Lang \cite{Lang} and Lang-Wenger \cite{LangWenger}. 
We now quickly review Lang-Wenger's currents,  
discuss their relation to Ambrosio--Kirchheim's
currents and state some convergence results in this setting.

Let $(X, \dist)$ be a non-necessarily complete metric space, let us denote by $ \Lipbs(X)$ the space of Lipschitz functions with bounded support in $X$, and $\Liploc(X)$ the space of functions which are Lipschitz  on bounded sets of $X$.

Recall that an \emph{$n$-dimensional metric functional} 
in the sense of \cite{LangWenger} is a multilinear function 
$T \colon \Lipbs(X) \times [\Liploc(X)]^n \to \setR$ such that 
\begin{itemize}
    \item $T$ is continuous: for any sequence 
    $\pi^i \in [\Liploc(X)]^n$ converging pointwisely to $\pi \in [\Liploc(X)]^n$ with $\sup_{i,j}\Lip(\pi^i_j|_A) < \infty$ for any bounded $A \subset X$, it holds
\begin{equation*}
T(f,\pi^i) \to T(f,\pi);
\end{equation*}
\item $T$ is local: if there exists $j$ such that $\pi_j$ is constant on a neighborhood of 
$\{f \neq 0\}$, then $T(f,\pi) = 0$.
\end{itemize}

For every open set $V \subset X$, the mass of $T$ in $V$ is defined as the possibly infinite quantity
\[
 \mass_V(T) := \sup \sum_{\lambda \in \Lambda} T(f_\lambda,\pi_\lambda),
\]
where the supremum is taken over all finite families 
$\bigl((f_\lambda,\pi_\lambda)\bigr)_{\lambda \in \Lambda}$ such that $(f_\lambda,\pi_\lambda) = (f_\lambda,\pi_{\lambda,1},\dots,\pi_{\lambda,n}) \in \Lipbs(X) \times [\Lip_1(X)]^n$, $\spt (f_\lambda) \subset V$, and $\sum_{\lambda \in \Lambda}|f_\lambda| \le 1$. Here $\Lip_1(X) \subset \Lip(X)$ denotes the subset of all $1$-Lipschitz functions.

\begin{definition}[\cite{LangWenger}]\label{def:loc-fin-mass}
For $n \ge 0$, an \emph{$n$-dimensional metric current with locally finite mass} is an
$n$-dimensional metric functional $T$ as defined above, satisfying
that for every bounded open set $V \subset X$ and every $\epsilon > 0$ there exists a compact set $C \subset V$ such that $\mass_V(T) < \infty$ and $\mass_{V \setminus C}(T) < \epsilon$.  We denote by $\LWc{X}{n}$ the vector space of all $n$-dimensional metric currents with locally finite mass. 
\end{definition}

If $T \in \LWc{X}{n}$, the set function $\|T\| \colon 2^X \to [0,\infty]$ given by
\[
 \|T\|(A) := \inf\left\{\mass_V(T): \text{$V \subset X$ is open, $A \subset V$}\right\}
\]
is an outer measure and, if $V$ is open, then $\|T\|(V) = \mass_V(T)$. As before, $\mass(T) := \|T\|(X)$, $\spt (T) := \spt\|T\|$, and $\|T\|(X \setminus \spt ( T)) = 0$.

Every $T \in \LWc{X}{n}$ can be extended to a multilinear function 
\[
T \colon \mathcal{B}_{bs}^\infty(X) \times [\Liploc(X)]^n \to \setR, 
\]
where $\mathcal{B}_{bs}^\infty(X)$ denotes the set of all bounded Borel functions with bounded support defined on $X$.
This extension is continuous and 
satisfies 
\begin{equation*}
        |T(f, \pi)| \leq \Lip(\pi_1|_{\spt(f)})\cdots \Lip(\pi_n|_{\spt(f)} )\int_X |f|\, \di \norm{T}.
    \end{equation*}
Given $T \in \LWc{X}{n}$ and $(g,\tau) \in \mathcal{B}_{bs}^\infty (X) \times [\Liploc(X)]^k$ with $0 \le k \le n$ and any map $\varphi\in\Liploc(X,X')$ such that
$\varphi^{-1}(B)$ is bounded whenever $B\subset X'$
is bounded,
 the restriction $T \rstr (g,\tau) \in \LWc{X}{n-k}$ and 
 the push-forward 
of $T$ by $\varphi$,
$\varphi_\#T \in \LWc{X'}{n}$, are defined in the same way as for Ambrosio-Kirchheim currents. In particular, for a
Borel set $A \subset X$, the restriction of $T$ to $A$ is a current $T \rstr A \in \LWc{X}{n}$ that satisfies 
\begin{equation}\label{eq-massRestriction}
||T   \rstr A|| = ||T ||  \rstr A.
\end{equation}

For $n \ge 1$, the functional $\partial T$ is defined by 
\[
\partial T(f,\pi_1,\dots,\pi_{n-1}) = T(\sigma,f,\pi_1,\dots,\pi_{n-1}),
\]
where $\sigma \in \Lipbs(X)$ is any function satisfying $\sigma|_{\spt(f)}=1$, and  $n$-dimensional local \emph{normal} currents are defined as
\[
\LWnc{X}{n} = \left\{T \in \LWc{X}{n}: \partial T \in \LWc{X}{n-1}\right\}.
\]
For $n=0$, we set  $\LWnc{X}{0} = \LWc{X}{0}$. 

A current $T\in\LWc{X}{n}$ is \emph{locally integer rectifiable}
if $\|T\|$ is concentrated on a countably $n$-rectifiable Borel set
and, for every bounded Borel set $B\subset X$ and every
$\Psi\in\Lip(X,\setR^n)$,
\[
 \Psi_\#(T\rstr B)=\curr{\theta}
 \qquad\text{for some }\theta\in L^1(\setR^n,\mathbb Z).
\]
We denote the space of such currents by $\LWirc{X}{n}$.
For currents in $\LWc{X}{n}$, the rectifiable concentration
requirement is equivalent to the compact rectifiable approximation
condition in \cite[Definition~2.12(i)]{LangWenger}.

For $n \ge 1$, the abelian group of all $T \in \LWirc{X}{n}$ such that $\partial T\in\LWirc{X}{n-1}$ is denoted by 
 $\LWic{X}{n}$, and for $n=0$ we set $\LWic{X}{0}= \LWirc{X}{0}$. 
 Elements in $\LWic{X}{n}$ are called locally integral currents.
 Similar to Ambrosio-Kirchheim's currents, one has 
 \[
 \LWic{X}{n} = \LWirc{X}{n} \cap \LWnc{X}{n}.
 \]

\subsubsection{Relation between Ambrosio--Kirchheim and Lang--Wenger's currents}

The proof of the next proposition is contained in  \cite[Sect.\ 2.5]{LangWenger}.

\begin{proposition}\label{prop-AKvsLW}
Assume that $(X,\dist)$ is complete.  
\begin{itemize}
    \item If $T \in \LWc{X}{n}$ and $g \in \mathcal {B}_{bs}^\infty(X)$,  then 
 $T_g:=T\rstr g: \mathcal D^n(X) \to \setR$ given by 
\[
T_g\: (f,\pi) = T(f g,\pi)
\] 
is an element of $\bM_n(X)$. 
If $n \ge 1$, $g \in \Lipbs(X)$ and $T \rstr g \in \LWnc{X}{n}$, then 
$T_g  \in \bN_n(X)$ and 
$\|\partial T_g \| = \|\partial(T \rstr g)\|$. 
\item Conversely, given an Ambrosio--Kirchheim current $T' \in \bM_n(X)$, by setting  
\begin{equation*}
T(f,\pi) := T'(f,\pi'),
\end{equation*}
for $(f,\pi) \in  \mathcal{B}_{bs}^\infty(X) \times [\Liploc(X)]^n$ and any $\pi' \in [\Lip(X)]^n$ with $\pi_i'|_{\spt f} = \pi_i|_{\spt f}$, 
one obtains a local current $T \in \LWc{X}{n}$ for which  $\|T\| = \|T'\|$. In case $n \ge 1$, it follows that 
\[
\partial T(f,\pi) = \partial T'(f,\pi'),
\]
for $(f,\pi) \in  \mathcal{B}_{bs}^\infty(X) \times [\Liploc(X)]^{n-1}$ and $\pi' \in [\Lip(X)]^{n-1}$ with 
$\pi_i'|_{\spt f} = \pi_i|_{\spt f}$. In particular, if $T' \in \bN_n(X)$, then $T \in \LWnc{X}{n}$ and $\|\partial  T\| = \|\partial T'\|$. 
\end{itemize}
\end{proposition}

\subsubsection{Local flat topology and pGH convergence}

Analogously to the previous section, one can define a notion of convergence for locally integral currents, as well as a result that relates the pointed Gromov–Hausdorff limit of a sequence to its local flat limit whenever both limits exist.

\begin{definition}\label{def:locFlatTop}
    Let $W$ be a metric space,   $T_i \in \LWic{W}{n}$, $i \in \mathbb N$, converges to $T  \in \LWic{W}{n}$ in the local flat topology if 
for any bounded closed set $B \subset W$ there exist  $S_i \in \LWic{W}{n+1}$
such that  $$\norm{ T -T_i - \partial S_i}(B) +  \norm{S_i}(B) \to 0.$$
\end{definition}

Local flat convergence is preserved under taking boundaries.
Indeed, fix a bounded closed set $B\subset W$ and choose
$S_i\in\LWic{W}{n+1}$ as in
Definition~\ref{def:locFlatTop}. Set
\[
 A_i:=T-T_i-\partial S_i.
\]
Then $\|A_i\|(B)\to0$ and
\[
 \partial T-\partial T_i=\partial A_i.
\]
Using $A_i$ as the filling currents in the definition
proves that $\partial T_i\to\partial T$ locally flat.

 \begin{theorem}[Theorem 1.2  and  Proposition 1.3 in \cite{LangWenger}]\label{thm-compactnessLW}
Let $(X_i,\dist_i,x_i)$ be complete pointed metric spaces and   $T_i \in \LWic{X_i}{n}$
such that 
$$\sup_{i \in \mathbb N} \Big  \{ ||  T_i || (B_r(x_i))  +   ||\partial T_i|| (B_r(x_i)) \Big  \}  < \infty$$
for all $r>0$. Then there exist a complete pointed metric space $(W, \dist_W, w_0)$, a subsequence $(i_j)_{j \in \mathbb N}$ and
isometric embeddings 
\[
\varphi_j: X_{i_j} \to W \quad \forall \, j \in \mathbb N, \qquad 
\varphi_j(x_{i_j})  \to w_0,
\]
such that  $\varphi_{j \sharp} T_{i_j} \to T$  in  the local flat topology to some $T \in \LWic{W}{n}$.

Furthermore, if $(X_i,\dist_i,x_i)$ converges in the pGH topology to $(X,\dist,x)$ and $(X,\dist)$ is proper, 
there exists an isometric embedding 
\[
\psi :  \spt(T) \cup \{w_0\} \to  X
\qquad \text{such that}\quad \psi(w_0)=x.
\]
\end{theorem}

\begin{lemma}[Common realization for pGH and local flat convergence]\label{lem:common-pgh-localflat}
Let
\[
    (X_i,\dist_i,x_i)
    \to
    (X,\dist,x)
\]
in the pointed Gromov--Hausdorff topology, and assume that all the spaces
are proper. Let
\[
    T_i\in\LWic{X_i}{N}
\]
satisfy, for every $R>0$,
\begin{equation}\label{eq:common-realization-mass-bound}
    \sup_i
    \left\{
        \|T_i\|(B_R(x_i))
        +
        \|\partial T_i\|(B_R(x_i))
    \right\}
    <\infty.
\end{equation}
Then, after passing to a subsequence, there exist a complete metric space
$(Z,\dist_Z)$, isometric embeddings
\[
    \iota_i:X_i\to Z,
    \qquad
    \iota:X\to Z,
\]
and a current
\[
    \widetilde T\in\LWic{Z}{N}
\]
such that
\begin{equation}\label{eq:common-realization-basepoints}
    \iota_i(x_i)\to\iota(x),
\end{equation}
the sets $\iota_i(X_i)$ converge locally in Hausdorff distance to
$\iota(X)$, and
\begin{equation}\label{eq:common-realization-localflat}
    \iota_{i\sharp}T_i
    \to
    \widetilde T
    \qquad\text{locally flat in }Z.
\end{equation}
Moreover,
\begin{equation}\label{eq:common-realization-support}
    \spt(\widetilde T)\subset\iota(X).
\end{equation}
\end{lemma}

\begin{proof}
Choose a complete metric space $Y$ realizing the pointed
Gromov--Hausdorff convergence, with isometric embeddings
\[
    j_i:X_i\to Y,
    \qquad
    j:X\to Y,
\]
such that $j_i(x_i)\to j(x)$ and $j_i(X_i)$ converges locally
in Hausdorff distance to $j(X)$. We may choose this realization
so that, for every $R>0$,
\[
    j(\overline B_R(x))
    \cup
    \bigcup_i j_i(\overline B_R(x_i))
\]
has compact closure in $Y$.

Let $J:Y\to Z:=\ell^\infty(Y)$ be a Kuratowski-Fr\'echet
isometric embedding, and set
\[
    \iota_i:=J\circ j_i,
    \qquad
    \iota:=J\circ j,
    \qquad
    z_0:=\iota(x),
    \qquad
    S_i:=\iota_{i\sharp}T_i.
\]
The basepoint and local Hausdorff convergence are preserved
by this embedding. Moreover,
\eqref{eq:common-realization-mass-bound} and
$\iota_i(x_i)\to z_0$ imply that, for every $R>0$,
\[
    \sup_i
    \bigl\{
        \|S_i\|(B_R(z_0))
        +  \|\partial S_i\|(B_R(z_0))
    \bigr\}<\infty.
\]
The supports of the $S_i$ in each fixed bounded ball are
contained in a common compact subset of $Z$.

We first localize the currents. Set
$\rho(z):=\dist_Z(z,z_0)$.
By the slicing theorem, Fatou's lemma and successive
subsequence extraction, we can choose radii
$R_m\in(m,m+1)$ such that, along a diagonal subsequence,
\[
    S_i^m:=S_i\rstr B_{R_m}(z_0)\in\bI_N(Z),
    \qquad
    \sup_i
    \bigl\{
        \mass(S_i^m)+\mass(\partial S_i^m)
    \bigr\}<\infty
\]
for every $m$. Here we use
\[
    \partial(S_i\rstr\{\rho<R\})
    =
    (\partial S_i)\rstr\{\rho<R\}
    +
    \langle S_i,\rho,R\rangle
\]
and the estimate
\[
    \int_m^{m+1}
    \mass\bigl(\langle S_i,\rho,R\rangle\bigr)\,dR
    \leq
    \|S_i\|(B_{m+1}(z_0)).
\]
As usual, finitely many initial indices may be discarded
for each fixed $m$.

For each $m$, the currents $S_i^m$ have uniformly bounded
mass and boundary mass and are supported in a common
compact set. The compactness and closure theorems for
integral currents
\cite[Theorems~5.2 and~8.5]{AmbrosioKirchheim00},
followed by a diagonal extraction, therefore give currents
$Q_m\in\bI_N(Z)$ such that
\[
    S_i^m\rightharpoonup Q_m
    \qquad\text{for every }m.
\]
If $\ell>m$, the currents $S_i^\ell$ and $S_i^m$ agree
on $B_{R_m}(z_0)$. Testing with functions supported
strictly inside this ball and passing to the weak limits
shows that
\[
    Q_\ell\rstr B_{R_m}(z_0)
    =
    Q_m\rstr B_{R_m}(z_0).
\]
Consequently the $Q_m$ define a locally integral current
$\widetilde T\in\LWic{Z}{N}$ satisfying
\[
    \widetilde T\rstr B_{R_m}(z_0)
    =
    Q_m\rstr B_{R_m}(z_0)
    \qquad\text{for every }m.
\]
Indeed, on each bounded set this current agrees with
one of the integral currents $Q_m$, with the agreement
holding on a neighborhood of that set.

We now use the choice of the ambient space.
Since $Z$ is a Banach space, Wenger's weak-to-flat
convergence theorem
\cite[Theorem~1.4]{WengerFlat} applies to each of the uniformly bounded sequences
$(S_i^m)_i$. Hence
\[
    \dist_F^Z(S_i^m,Q_m)\to0
    \qquad\text{for every }m.
\]
Let $B\subset Z$ be bounded and closed, and choose $m$
such that $B\subset B_{R_m}(z_0)$. There exist
$A_i^m\in\bI_N(Z)$ and $V_i^m\in\bI_{N+1}(Z)$ such that
\[
    Q_m-S_i^m=A_i^m+\partial V_i^m,
    \qquad
    \mass(A_i^m)+\mass(V_i^m)\to0.
\]
Since $\widetilde T$ and $Q_m$, and respectively $S_i$
and $S_i^m$, agree on $B_{R_m}(z_0)$, we obtain
\[
    \|\widetilde T-S_i-\partial V_i^m\|(B)
    +
    \|V_i^m\|(B)
    \leq
    \mass(A_i^m)+\mass(V_i^m)
    \to0.
\]
This proves \eqref{eq:common-realization-localflat}.

Finally, the local Hausdorff convergence implies that
every point of $Z\setminus\iota(X)$ has a neighborhood
disjoint from $\iota_i(X_i)$ for all sufficiently large
$i$. Weak convergence of the localized currents shows
that every $Q_m$ vanishes on such a neighborhood.
Therefore $\spt(\widetilde T)\subset\iota(X)$, which is \eqref{eq:common-realization-support}.
\end{proof}

\section{Bilipschitz charts, regularity properties for m.m.\ spaces and local degree}\label{sec-BilipProperties}

\subsection{Bilipschitz decomposition of charts in a non-collapsed $\RCD(K,N)$ space}\label{SS:CritHarm}

The aim of this subsection is to prove a full-measure biLipschitz decomposition, Proposition \ref{prop:BilipPartition},  for  
$\delta$-splitting maps which are topological charts around a regular point of a non-collapsed $\RCD(K,N)$ space.  The argument is local and quantitative. Thus, in Lemma \ref{lem:fixedgoodset}, we first find a good set that will be used in the decomposition. A key point in the proof of this lemma is that the normalizing matrix given by the transformation theorem (see \cite[Proposition 3.13]{BrueNaberSemola20} for a proof in the $\RCD$ setting, after \cite{CJN:21}) is allowed to depend on the center and on the scale.

\begin{lemma}\label{lem:fixedgoodset}
For every $N\in\setN$ and every sufficiently small
$\varepsilon_0>0$, there exist
\[
    \eta_{\rm in}=\eta_{\rm in}(N,\varepsilon_0)>0,
    \qquad
    \delta_0=\delta_0(N,\varepsilon_0)>0
\]
with the following property.
Let $(X,\dist,\haus^N)$ be a non-collapsed
$\RCD(-\delta_0^2(N-1),N)$ space and assume that
$B_{16}(p)$ is $(N,\delta_0)$-symmetric.
Let
\[
    u:B_8(p)\to\setR^N
\]
be an $\eta_{\rm in}$-splitting map.
Then, for every $y\in B_1(p)$ and every $0<s<1/32$,
there exists a Borel set $G_{y,s}\subset B_s(y)$ such that
\begin{equation}\label{eq:good-set-interior}
    \haus^N(B_s(y)\setminus G_{y,s})
    \leq \varepsilon_0 s^N,
\end{equation}
and $u|_{G_{y,s}}$ is biLipschitz onto its image.
\end{lemma}

\begin{proof}
We distinguish the splitting error of the original map
from the error obtained after applying the transformation
theorem.

Fix $0<\varepsilon_0<\omega_N/8$.
Let $\kappa=\kappa(N)>0$ be sufficiently small that
the full-dimensional splitting-to-isometry estimate
\cite[Remark~3.10]{BrueNaberSemola20}, after rescaling,
has the following consequence: if $w$ is a
$\kappa$-splitting map on $B_t(z)$, the ball
$B_{3t/2}(z)$ is $(N,\kappa)$-symmetric, and the
rescaled curvature lower bound is sufficiently small,
then
\begin{equation}\label{eq:good-set-additive-isometry}
    \bigl|
        |w(a)-w(b)|-\dist(a,b)
    \bigr|
    \leq \frac{t}{8}
    \qquad\text{for all }a,b\in B_t(z).
\end{equation}

Choose an output error $\eta_{\rm out}>0$ sufficiently
small in terms of $N$ and $\varepsilon_0$.
More precisely, let $C_0=C_0(N)\geq1$ be a dimensional
constant large enough for the interior Hessian estimate
used below, and set $\alpha:=C_0\eta_{\rm out}$.
We choose $\eta_{\rm out}$ so that $\alpha<1$,
\begin{equation}\label{eq:good-set-output-smallness}
    C_{\rm prop}(N)\alpha^{1/4}<\kappa,
    \qquad
    8^N C_{\rm max}(N)\alpha^{1/2}
    \leq\varepsilon_0,
\end{equation}
where $C_{\rm prop}$ and $C_{\rm max}$ are the constants
in the propagation and maximal-function estimates below.

Next choose $\delta_{\rm tr}>0$ below the threshold
in \cite[Proposition~3.13]{BrueNaberSemola20}
corresponding to output error $\eta_{\rm out}$.
We also require
\[
    \delta_{\rm tr}<1,
    \qquad
    \delta_{\rm tr}^2<\kappa.
\]
Choose the input error $\eta_{\rm in}>0$ so that
\begin{equation}\label{eq:good-set-input-smallness}
    C_{\rm rec}(N)\eta_{\rm in}<\delta_{\rm tr},
\end{equation}
where $C_{\rm rec}(N)$ is the dimensional constant
arising when restricting the splitting estimates
from $B_8(p)$ to $B_2(y)$, for $y\in B_1(p)$.

Finally, choose $\delta_0>0$ sufficiently small, with
$\delta_0\leq\delta_{\rm tr}$, so that
\begin{equation}\label{eq:all-scale-symmetry-y}
    B_t(z)\text{ is $(N,\delta_{\rm tr}^2)$-symmetric}
    \qquad
    \text{for all }z\in B_2(p),\quad 0<t\leq1.
\end{equation}
This follows from volume convergence,
Bishop--Gromov monotonicity and quantitative volume
rigidity; see, for example,
\cite[Theorem~2.13]{KapMonGT}.
Indeed, the almost Euclidean geometry of $B_{16}(p)$
gives almost maximal volume ratios on fixed-radius
balls centered in $B_2(p)$. Bishop--Gromov comparison
propagates these lower bounds to smaller scales,
and volume rigidity gives
\eqref{eq:all-scale-symmetry-y}.
All constants are uniform in the center and the scale.

Fix $y\in B_1(p)$. Since
$B_2(y)\subset B_3(p)\subset B_8(p)$,
volume comparison and the splitting estimates for $u$
show that
\begin{equation}\label{eq:recentered-almost-splitting}
    u:B_2(y)\to\setR^N
\end{equation}
is a $C_{\rm rec}(N)\eta_{\rm in}$-splitting map.
Here harmonicity and the gradient bound are preserved
under restriction, while the averaged Hessian and
Gram-matrix estimates follow from volume comparability.
By \eqref{eq:good-set-input-smallness}, this map satisfies
the input requirement of the transformation theorem.

Let $0<s<1/32$. Apply
\cite[Proposition~3.13]{BrueNaberSemola20}
with $r_0=32s$ and output error $\eta_{\rm out}$.
Its curvature and all-scale symmetry assumptions follow
from $\delta_0\leq\delta_{\rm tr}$ and
\eqref{eq:all-scale-symmetry-y}.
We obtain an invertible lower triangular matrix
\[
    A_{y,s}\in{\rm GL}(N)
\]
such that
\[
    v:=A_{y,s}u
\]
is an $\eta_{\rm out}$-almost splitting map on
$B_{32s}(y)$.
The matrix is allowed to depend on both $y$ and $s$.
Since $u$ is harmonic, so is $v$.
The interior Hessian estimate following
\cite[Definition~3.11]{BrueNaberSemola20}
therefore makes $v$ an $\alpha$-splitting map on
$B_{16s}(y)$, after increasing $C_0(N)$ if necessary.

Set
\[
    \dist_s:=(8s)^{-1}\dist,
    \qquad
    \mu_s:=(8s)^{-N}\haus^N,
    \qquad
    w:=(8s)^{-1}\bigl(v-v(y)\bigr).
\]
In the remainder of the argument, balls, gradients
and Hessians are taken with respect to the rescaled
metric $\dist_s$.
Then $w$ is an $\alpha$-splitting map on $B_2(y)$.
In particular, setting
\[
    H:=\sum_{a=1}^N|\Hess w_a|^2,
\]
we have
\[
    \int_{B_2(y)}H\,d\mu_s\leq C(N)\alpha.
\]
The volume factor is uniformly bounded by
Bishop--Gromov comparison.

For $z\in B_{1/8}(y)$, define the truncated maximal
function
\[
    \mathcal M H(z)
    :=
    \sup_{0<t<1}
    \fint_{B_t(z)}H\,d\mu_s,
\]
and set
\[
    G:=
    \left\{
        z\in B_{1/8}(y):
        \mathcal M H(z)\leq\alpha^{1/2}
    \right\}.
\]
All balls in this definition lie in $B_2(y)$.
The local doubling property and the weak $(1,1)$
maximal inequality give
\begin{equation}\label{eq:max-good-measure}
    \mu_s(B_{1/8}(y)\setminus G)
    \leq
    C(N)\alpha^{-1/2}
    \int_{B_2(y)}H\,d\mu_s
    \leq C_{\rm max}(N)\alpha^{1/2}.
\end{equation}

For every $z\in G$ and $0<t<1$,
\[
    t\fint_{B_t(z)}H\,d\mu_s
    \leq t\alpha^{1/2}
    \leq\alpha^{1/2}.
\]
Thus \cite[Lemma~4.16]{BrueNaberSemola20}, applied
with an arbitrary positive lower cutoff and then
allowing that cutoff to decrease to zero, shows that
$w$ is a
$C_{\rm prop}(N)\alpha^{1/4}$-splitting map on
$B_t(z)$ for every $0<t<1$.

Let $z,z'\in G$ be distinct, and put
\[
    d_*:=\dist_s(z,z'),
    \qquad
    t:=2d_*.
\]
Since $z,z'\in B_{1/8}(y)$, we have $t<1/2$.
The ball $B_{3t/2}(z)$, in the rescaled metric,
corresponds to a ball of radius $12st<1$ in the
original metric, centered at a point of $B_2(p)$.
It is therefore $(N,\delta_{\rm tr}^2)$-symmetric
by \eqref{eq:all-scale-symmetry-y}.
Together with
\eqref{eq:good-set-output-smallness}
and the curvature bound, this allows us to apply
\eqref{eq:good-set-additive-isometry} to $w$ on
$B_t(z)$. We obtain
\[
    \bigl|
        |w(z)-w(z')|-d_*
    \bigr|
    \leq \frac{t}{8}
    =\frac{d_*}{4}.
\]
Consequently,
\[
    \frac34\dist_s(z,z')
    \leq |w(z)-w(z')|
    \leq \frac54\dist_s(z,z')
    \qquad\text{for all }z,z'\in G.
\]
Hence $w|_G$ is biLipschitz (this argument showing the biLipschitz property of $w|_G$ is inspired by \cite[Theorem~8.4(ii)]{BrueNaberSemola20} and \cite[Theorem~6.5]{MondinoNaber19}).

Regard $G$ as a subset $G_{y,s}$ of the original
space. Since $B_{1/8}^{\dist_s}(y)=B_s^{\dist}(y)$,
\eqref{eq:max-good-measure} and
\eqref{eq:good-set-output-smallness} imply
\[
    \haus^N(B_s(y)\setminus G_{y,s})
    \leq
    (8s)^N C_{\rm max}(N)\alpha^{1/2}
    \leq\varepsilon_0s^N.
\]
Scaling back shows that $v|_{G_{y,s}}$ is biLipschitz.
Since $A_{y,s}$ is invertible, the same holds for
$u|_{G_{y,s}}$.
\end{proof}

We now obtain the full-measure decomposition. 

\begin{proposition}\label{prop:BilipPartition}
For every $N\in\setN$, $N\geq1$, there exists $\delta(N)>0$ with the following property.  Let $(X,\dist,\haus^N)$ be a non-collapsed $\RCD(-\delta(N-1),N)$ space and assume that $B_{16}(p)$ is an $(N,\delta)$-symmetric ball.  Let
\[
 u:B_8(p)\to\setR^N
\]
be the almost-splitting map given by
Theorem~\ref{thm:biholder}, with splitting error at most
$\eta_{\rm in}(N,\varepsilon_*)$ from
Lemma~\ref{lem:fixedgoodset}, where
\[
    \varepsilon_*:=\min\{1/100,\omega_N/100\}.
\]
 Then there exists a countable family of pairwise disjoint compact sets $K_i\subset B_1(p)$ such that
\begin{equation}\label{eq:fixed-chart-full-measure}
 \haus^N\left(B_1(p)\setminus\bigcup_{i\in\setN}K_i\right)=0
\end{equation}
and $u|_{K_i}$ is biLipschitz onto its image for every $i\in\setN$.
\end{proposition}

\begin{proof}
Set $\varepsilon_0:=\varepsilon_*$, and let
$\eta_{\rm in}$ and $\delta_0$ be given by
Lemma~\ref{lem:fixedgoodset}.
Choose the geometric parameter $\delta$ in the statement
small enough that
\[
    \delta\leq\min\{\delta_0,\delta_0^2\},
\]
and that Theorem~\ref{thm:biholder} supplies an almost-splitting
map with error at most $\eta_{\rm in}$.
Then the curvature, symmetry and splitting assumptions
of Lemma~\ref{lem:fixedgoodset} are satisfied.

  We first prove that every Borel set $E\subset B_1(p)$ with $\haus^N(E)>0$ contains a compact subset of positive measure on which $u$ is biLipschitz.
Let
\[
y\in E\cap\mathcal R_N
\]
which is simultaneously  an \(N\)-regular point and an \(\haus^N\)-density point of \(E\), i.e.:
\begin{align}
\lim_{s\downarrow0}\frac{\haus^N(B_s(y))}{\omega_Ns^N}=1,  \label{eq:density-interior-ball}\\
 \lim_{s\downarrow0} \frac{\haus^N(B_s(y)\setminus E)}{s^N}=0. \label{eq:density-interior-E}
\end{align}
The existence of such $y$ follows since both properties hold at \(\haus^N\)-a.e. point of \(E\).
Since $B_1(p)$ is open, for $s>0$ sufficiently small $B_s(y)\subset B_1(p)$.  We can also assume, using \eqref{eq:density-interior-ball}--\eqref{eq:density-interior-E}, that
\[
 \haus^N(B_s(y))\geq\frac34\omega_Ns^N,
 \qquad
 \haus^N(B_s(y)\setminus E)\leq\varepsilon_0s^N.
\]
Let $G_{y,s}$ be given by Lemma~\ref{lem:fixedgoodset}.  Then
\begin{align*}
 \haus^N(E\cap G_{y,s})
 &\geq \haus^N(B_s(y))
 -\haus^N(B_s(y)\setminus E)
 -\haus^N(B_s(y)\setminus G_{y,s})\\
 &\geq\left(\frac34\omega_N-2\varepsilon_0\right)s^N>0.
\end{align*}
By inner regularity, there exists a compact set
$K\subset E\cap G_{y,s}$
with $\haus^N(K)>0$.  The restriction $u|_K$ is biLipschitz.

We have therefore shown that every Borel subset of $B_1(p)$ with positive measure contains a compact subset with positive measure on which the fixed map $u$ is biLipschitz.  Since $\haus^N(B_1(p))<\infty$, consider a maximal family of pairwise disjoint compact sets
\[
 \{K_i\}_{i\in I},\qquad K_i\subset B_1(p),
\]
with $\haus^N(K_i)>0$ and $u|_{K_i}$ biLipschitz.  The index set $I$ is countable because the sets have positive measure and are pairwise disjoint in a finite-measure space.  If
\[
 \haus^N\left(B_1(p)\setminus\bigcup_{i\in I}K_i\right)>0,
\]
the property proved in the preceding paragraph, applied to the remainder, would produce another compact set with positive measure disjoint from every $K_i$ on which $u$ is biLipschitz, contradicting maximality.  This proves \eqref{eq:fixed-chart-full-measure}.
\end{proof}

\subsection{Some regularity properties for metric measure spaces}\label{Sec:RegPropmms}
We now give the precise definitions of (ETR), a.e.  quasi-convex regular set and (LDB). We start by collecting the properties provided by the preceding proposition in the following definition.

\begin{definition}[Essential topological regularity  (ETR)]\label{def:assumptions}
Let $(X,\dist,\haus^N)$ be a metric measure space and let $p\in X$.  We say that $p$ admits an \emph{essential topological regularity chart} if for every $\varepsilon>0$ and every $r^*>0$ there exist $r\in(0, r^*)$ and a Lipschitz homeomorphism
\[
 u:B_r(p)\to u(B_r(p))\subset\setR^N,
 \qquad u(p)=0,
\]
such that
\begin{enumerate}
\item $\Lip(u)\leq1+\varepsilon$;
\item $u(B_r(p))\subset \setR^N$ is open and $u(B_r(p))\supset B_{(1-2\varepsilon)r}^{\setR^N}(0)$;
\item $(1-\varepsilon)\omega_Nr^N
 \leq\haus^N(B_r(p))
 \leq(1+\varepsilon)\omega_Nr^N$;
\item there exists a countable family of pairwise disjoint compact sets $K_i\subset B_r(p)$ covering $B_r(p)$ up to a $\haus^N$-negligible set and such that $u|_{K_i}$ is biLipschitz onto for every $i$.
\end{enumerate}
We say that $(X,\dist,\haus^N)$ satisfies \emph{(ETR)} if every metric ball is connected and there exists a Borel set $\mathcal R\subset X$ of full $\haus^N$-measure such that every point of $\mathcal R$ admits such charts.  We call $\mathcal R$ the \emph{regular set} associated with (ETR).

Let $(X,\dist,\haus^N)$ satisfy  (ETR). We say that $(X,\dist,\haus^N)$ \emph{has no boundary in the (ETR) sense} if 
\begin{equation}\label{eq:SingCod>1ETR}
{\rm Codim_{Hauss}}(X\setminus \mathcal R)>1,
\end{equation} 
where ${\rm Codim_{Haus}}$ denotes the Hausdorff codimension.
	\end{definition}
    Notice that the notion of having no boundary for a space satisfying (ETR) is compatible with the analogous notion for non-collapsed $\RCD$ spaces, see \eqref{eq:SingCod>1}.

\begin{definition}[a.e.\ quasi-convex regular set]\label{def-ae-quasiconvex}
Let $(X,\dist,\haus^N)$ satisfy (ETR), with regular set $\mathcal R$.  We say that $\mathcal R$ is \emph{a.e. quasi-convex} if for every $x\in X$ there exists a constant $C_x\geq1$ such that, for $\haus^N\times\haus^N$-a.e.\ pair
\[
 (p,q)\in (\mathcal R\cap B_1(x))\times(\mathcal R\cap B_1(x)),
\]
there exists a rectifiable curve $\gamma:[0,1]\to\mathcal R$ joining $p$ to $q$ and satisfying
\begin{equation}\label{eq-ae-quasiconvex}
 \operatorname{Length}(\gamma)\leq C_x\,\dist(p,q).
\end{equation}
The exceptional set of pairs is allowed to depend on $x$.
\end{definition}

\begin{remark}\label{rem-aeqc-local}
The fixed radius $1$ in Definition~\ref{def-ae-quasiconvex} is only a normalization.  In the proof of Proposition~\ref{prop:locrep2} the property is applied to pairs in an arbitrarily small ball $B_\rho(y)\subset B_1(y)$.  If $\operatorname{Length}(\gamma)\le C_y\dist(p,q)$ and $p,q\in B_\rho(y)$, then
for every $z\in\gamma([0,1])$,
\[
\dist(z,y)\le\dist(z,p)+\dist(p,y)
\le\operatorname{Length}(\gamma)+\rho
 \le C_y \dist(p,q) +  \rho \leq (2C_y +1 )\rho,
\]
 and thus 
\begin{equation}\label{eq:curve-Ball}
 \gamma([0,1])\subset B_{(1+2C_y)\rho}(y).
\end{equation}
\end{remark}

\begin{definition}\label{def:assumptions2}(LDB)
	Let $(X,\dist,\haus^N)$ be a metric measure space, where $1\le N<\infty$ is a natural number. We say that $(X,\dist,\haus^N)$ satisfies (LDB) (for \emph{lower density bound}) if
 \begin{equation}\label{eq:ASSM2}
 \liminf_{r\to 0^+} \frac{\haus^N(B _r(x))}{r^N}>0, \quad \text{for all } x\in X.
 \end{equation}
\end{definition}

\begin{lemma}\label{lem:SuffASSM2}
Let $(X,\dist,\haus^N)$ be a metric measure space.
Assume that there is a dense set $E\subset X$ such that,
for every $x\in X$, there exist constants $r_x,c_x>0$
satisfying
\begin{equation}\label{eq:LowerHalfors}
 \haus^N(B_r(p))\geq c_x r^N
 \qquad
 \text{for all }p\in E\cap B_{r_x}(x)
 \text{ and }0<r<r_x.
\end{equation}
Then $(X,\dist,\haus^N)$ satisfies (LDB).
\end{lemma}

\begin{proof}
Fix $x\in X$ and $0<r<r_x$.
By density of $E$, choose $p\in E\cap B_{r/2}(x)$.
Since $B_{r/2}(p)\subset B_r(x)$,
\eqref{eq:LowerHalfors} gives
\[
 \haus^N(B_r(x))
 \geq \haus^N(B_{r/2}(p))
 \geq c_x(r/2)^N.
\]
Thus
\[
 \liminf_{r\downarrow0}
 \frac{\haus^N(B_r(x))}{r^N}
 \geq 2^{-N}c_x>0,
\]
which proves (LDB).
\end{proof}

\begin{proposition}\label{prop-ex-assms}
The following classes of spaces satisfy both (ETR) and (LDB):
\begin{itemize}
\item non-collapsed $\RCD(K,N)$ spaces  $(X,\dist,\haus^N)$, for some $K\in \setR$ and $N\in \setN$.

\item Non-collapsed strong Kato limit
spaces, in the sense of Carron-Mondello-Tewodrose (see Definition \ref{def:strongKatoLimit}).
\end{itemize}
Moreover, if $(X,\dist, \haus^N)$ belongs to one of two classes above and has no boundary in the sense of \eqref{eq:SingCod>1ETR}, then the regular set can be chosen to be a.e.\ quasi-convex in the sense of Definition~\ref{def-ae-quasiconvex}.
\end{proposition}

\begin{proof}
For \emph{non-collapsed $\RCD$ spaces}, fix an $N$-regular point and rescale at a sufficiently small radius.  Theorem~\ref{thm:biholder} gives the Lipschitz homeomorphism, the image estimate and the volume estimate, while Proposition~\ref{prop:BilipPartition} gives the full-measure decomposition for that same map.
Since the $N$-regular set has full $\haus^N$-measure and $\RCD$ spaces are geodesic, (ETR) follows.

Let $\mathcal R$ denote the $N$-regular set.
To verify (LDB), fix $x\in X$.
Bishop--Gromov comparison, applied with curvature lower
bound $\min\{K,0\}$, gives a constant $c=c(K,N)>0$
such that
\[
 \haus^N(B_r(p))
 \geq c\,\haus^N(B_1(x))\,r^N
 \qquad
 \text{for all }p\in B_1(x),\quad 0<r<1.
\]
Indeed, $B_1(x)\subset B_2(p)$, and volume comparison
gives
\[
 \frac{\haus^N(B_r(p))}{\haus^N(B_2(p))}\geq c\,r^N;
\]
see \cite[Thm.~2.3]{Sturm:II}.
Since $\mathcal R$ is dense and $\haus^N(B_1(x))>0$,
Lemma~\ref{lem:SuffASSM2} applies with
\[
 E=\mathcal R,\qquad r_x=1,\qquad
 c_x=c\,\haus^N(B_1(x)).
\]
Hence (LDB) holds.

Now assume that $X$ has no boundary.
We claim that $\mathcal R$ is a.e.\ quasi-convex
with constant $1$.

Choose a sequence $\varepsilon_j\downarrow0$ sufficiently small and set
$$
M_j:=\mathring{({\mathcal R}_N)_{\varepsilon_j}},
\qquad
B_j:=X\setminus M_j.
$$
By \cite[Theorem~4.10]{KapMonGT}, $\mathcal R\subset M_j$, the set $B_j$ is closed, and, since $X$ has no boundary,
$\haus^{N-1}(B_j)=0$.

Fix $p\in\mathcal R$. By \cite[Corollary~A.8]{KapMonGT}, for every $j$ and for $\haus^N$-a.e.\ $q\in\mathcal R$ there exists a minimizing geodesic from $p$ to $q$ contained in $M_j$. Moreover, for fixed $p$, the minimizing geodesic from $p$ to $q$ is unique for $\haus^N$-a.e.\ $q$, see for instance 
\cite[Lemma~4.1]{HanMondinoAngles}. Therefore, after intersecting the corresponding countably many
full-measure sets, for $\haus^N$-a.e.\ $q\in\mathcal R$ the unique
minimizing geodesic $\gamma_{p,q}$ is contained in every $M_j$.
We also have
\[
    \mathcal R=\bigcap_{j=1}^{\infty}M_j.
\]
Indeed, the inclusion
\[
    \mathcal R\subseteq\bigcap_{j=1}^{\infty}M_j
\]
follows from \cite[Theorem~4.10(1)]{KapMonGT}. Conversely, if
$x\in M_j$ for every $j$, then, by the definition of the quantitative
regular sets $({\mathcal R}_N)_{\varepsilon_j}$, every sufficiently
small rescaling at $x$ is $\varepsilon_j$-close to the Euclidean model.
Since $\varepsilon_j\downarrow0$, every tangent cone at $x$ is
isometric to $\setR^N$, and hence $x\in\mathcal R$.

It follows that
$$
\gamma_{p,q}\subset\mathcal R,
\qquad
\operatorname{Length}(\gamma_{p,q})=\dist(p,q).
$$
Fubini's theorem then gives the a.e.\ quasi-convexity property with constant $C_x=1$ for every $x\in X$.
\medskip

For \emph{a non-collapsed strong Kato limit space
$(X,\dist,\haus^N)$}, (LDB) follows from
\cite[Proposition~5.9]{CarronMondelloTewodrose2023}.

We verify (ETR).
The argument is the same as in
Lemma~\ref{lem:fixedgoodset} and
Proposition~\ref{prop:BilipPartition}, with the Kato
transformation and propagation estimates in place of
their Ricci-limit and $\RCD$ counterparts.
We give the details needed to ensure that the
biLipschitz pieces belong to the same coordinate map.

Fix an $N$-regular point $p$ and $\varepsilon>0$.
The quantitative harmonic-coordinate theorem and
volume convergence provide, at a sufficiently small
scale, a harmonic map $u$ defined on a neighborhood
of $\overline B_r(p)$ whose restriction to $B_r(p)$
is a homeomorphism satisfying
\[
 u(p)=0,\qquad
 \Lip(u|_{B_r(p)})\leq1+\varepsilon,\qquad
 B^{\mathbb R^N}_{(1-2\varepsilon)r}(0)
 \subset u(B_r(p)),
\]
and
\[
 (1-\varepsilon)\omega_Nr^N
 \leq\haus^N(B_r(p))
 \leq(1+\varepsilon)\omega_Nr^N.
\]
See \cite[Corollary~5.13 and Theorem~5.19]{CMT:AFM}
and, for complete manifolds under the present
integrability assumption,
\cite[Section~4.2, Remark~4.5, Theorem~5.11 and
Remark~5.12]{CarronMondelloTewodrose2023}.
We choose the original splitting error and the
Reifenberg error sufficiently small for the
argument below, before fixing $r$.

We claim that, for every $y\in B_r(p)$ and all
sufficiently small $s>0$, there is a Borel set
$G_{y,s}\subset B_s(y)$ such that
\begin{equation}\label{eq:kato-same-map-good-set}
 \haus^N(G_{y,s})\geq\tfrac12\haus^N(B_s(y)),
 \qquad
 u|_{G_{y,s}}\ \text{is biLipschitz}.
\end{equation}
Indeed, fix a sufficiently small output error
$\eta>0$.
Recenter $u$ first on a ball of fixed radius
comparable to $r$, contained in its original
domain. Volume comparability controls the
splitting error on this ball.
The Reifenberg estimate supplies the required
Euclidean closeness at every smaller scale.
The transformation theorem
\cite[Theorem~5.14]{CMT:AFM} then gives, for small
$s$, an invertible matrix $A_{y,s}$ such that
$A_{y,s}u$ is a sufficiently accurate splitting
on $B_{16s}(y)$.
Here the original splitting error is chosen below
the transformation threshold corresponding to
$\eta$.

By \cite[Remark~3.6]{CMT:AFM}, after restricting
to $B_{8s}(y)$ this is a reinforced splitting
with error at most $\eta$.
The propagation estimate
\cite[Proposition~3.12]{CMT:AFM}, followed by the
distance estimate
\cite[Proposition~3.10]{CMT:AFM}, gives a set
$G_{y,s}\subset B_s(y)$ on which $A_{y,s}u$ is
biLipschitz and
\[
 \haus^N(B_s(y)\setminus G_{y,s})
 \leq C\sqrt{\eta}\,\haus^N(B_s(y)).
\]
This is the prescribed-map argument
in the proof of \cite[Proposition~4.5]{CMT:AFM};
it does not require replacing $A_{y,s}u$ by a
new harmonic map.
Choose $\eta$ so that $C\sqrt{\eta}<1/2$.
Since $A_{y,s}$ is invertible, the same set is
biLipschitz for $u$, proving
\eqref{eq:kato-same-map-good-set}.
These estimates pass to Kato limits; see \cite[Theorem~3.8 and
Remark~3.6]{CMT:AFM}.
Their extension to the complete setting and the
weaker strong Kato assumption is explained in
\cite[Section~4.2 and Section~5.3]
{CarronMondelloTewodrose2023}.
Once established \eqref{eq:kato-same-map-good-set}, the  full-measure decomposition follows verbatim as in the end of the proof of Proposition~\ref{prop:BilipPartition}.
 Since the $N$-regular set has full $\haus^N$-measure
and $X$ is geodesic, every metric ball is connected.
Hence $X$ satisfies (ETR).
\smallskip

We finally verify the a.e.\ quasi-convexity of the regular set.
Let
\[
    \Theta_X(z):=
    \lim_{r\downarrow0}
    \frac{\haus^N(B_r(z))}{\omega_Nr^N},
    \qquad z\in X .
\]
By \cite[Proposition~5.9]{CarronMondelloTewodrose2023}, the limit exists
at every point and $0<\Theta_X(z)\leq1$.
By \cite[Proposition~5.5]{CMT:AFM}, the regular set satisfies
\begin{equation}\label{eq:kato-regular-density}
    \mathcal R
    =
    \{z\in X:\Theta_X(z)=1\}.
\end{equation}
Moreover, $\Theta_X$ is lower semicontinuous; see the proof of
\cite[Proposition~5.5]{CMT:AFM}. Choose a sequence
$\delta_j\downarrow0$ and set
\[
    M_j:=\{z\in X:\Theta_X(z)>1-\delta_j\},
    \qquad
    B_j:=X\setminus M_j .
\]
Then $M_j$ is open, $B_j$ is closed, and
\begin{equation}\label{eq:kato-reg-intersection}
    \mathcal R=\bigcap_{j=1}^{\infty}M_j.
\end{equation}
Furthermore, $B_j\subset X\setminus\mathcal R$ and therefore
\begin{equation}\label{eq:kato-Bj-codim2}
    \dim_{\rm Haus}(B_j)\leq N-2
\end{equation}
by \cite[Theorem~A.1]{CMT:AFM}.

We next use the auxiliary RCD structure given by
\cite[Corollary~B]{CarronMondelloTewodrose2023}. Notice that, after
possibly decreasing the time $t_0$ in the definition of the strong Kato
bound, we may assume that
\[
f(t_0)<\frac1{16N}.
\]
There exist a distance $\bar\dist$ and a measure $\bar\meas$ on $X$,
and constants $\bar K\geq0$, $\bar N<\infty$ and $L\geq1$, such that
\[
    (X,\bar\dist,\bar\meas)
    \quad\text{is an }\RCD(-\bar K,\bar N)\text{ space},
\]
and
\begin{equation}\label{eq:kato-RCD-comparison}
    \dist\leq\bar\dist\leq L\, \dist,
    \qquad
    \haus^N\leq\bar\meas\leq L\, \haus^N .
\end{equation}
Here we used that the reference measure of a non-collapsed strong Kato
limit is $\haus^N$; see again
\cite[Proposition~5.9]{CarronMondelloTewodrose2023}.

We claim that
\begin{equation}\label{eq:kato-codim-one-zero}
    \bar\meas_{-1}(B_j)=0
    \qquad\text{for every }j\in\mathbb N,
\end{equation}
where $\bar\meas_{-1}$ denotes the codimension-one outer measure
associated with $(\bar\dist,\bar\meas)$ in the sense of
\cite[Appendix~A]{KapMonGT}.

To prove the claim, for $\ell\in\mathbb N$ set
\[
    E_{j,\ell}
    :=
    \left\{
    z\in B_j:
    \haus^N(B_\rho(z))
    \leq2\omega_N\rho^N
    \text{ for every }0<\rho<\ell^{-1}
    \right\}.
\]
Since $\Theta_X(z)\leq1$ for every $z\in X$, we have
\[
    B_j=\bigcup_{\ell=1}^{\infty}E_{j,\ell}.
\]
Fix $j,\ell$. By \eqref{eq:kato-Bj-codim2},
\[
    \haus^{N-1}(E_{j,\ell})=0.
\]
Hence, given $\eta,\delta>0$, we can cover $E_{j,\ell}$ by balls
$B_{r_i}(z_i)$, with $z_i\in E_{j,\ell}$ and
\[
    r_i<
    \min\left\{
        \frac{\delta}{L},
        \frac{1}{L\ell}
    \right\},
    \qquad
    \sum_i r_i^{N-1}<\eta.
\]
By \eqref{eq:kato-RCD-comparison},
\[
    B_{r_i}^{\dist}(z_i)
    \subset
    B_{Lr_i}^{\bar\dist}(z_i),
\]
so that the balls $B_{Lr_i}^{\bar\dist}(z_i)$ cover
$E_{j,\ell}$. On the other hand,
\[
    B_{Lr_i}^{\bar\dist}(z_i)
    \subset
    B_{Lr_i}^{\dist}(z_i),
\]
because $\dist\leq\bar\dist$. Therefore, using again
\eqref{eq:kato-RCD-comparison} and the definition of
$E_{j,\ell}$,
\begin{align*}
    \bar\meas
    \big(B_{Lr_i}^{\bar\dist}(z_i)\big)
    &\leq
    L\,\haus^N
    \big(B_{Lr_i}^{\dist}(z_i)\big)\\
    &\leq
    2L\omega_N(Lr_i)^N.
\end{align*}
It follows that
\begin{align*}
    (\bar\meas_{-1})_{\delta}(E_{j,\ell})
    &\leq
    \sum_i
    (Lr_i)^{-1}
    \bar\meas
    \big(B_{Lr_i}^{\bar\dist}(z_i)\big)\\
    &\leq
    2\omega_N L^N
    \sum_i r_i^{N-1}\\
    &\leq
    2\omega_N L^N\eta.
\end{align*}
Since $\eta>0$ is arbitrary, by letting $\delta\downarrow0$, we obtain
$\bar\meas_{-1}(E_{j,\ell})=0$.
The countable subadditivity of $\bar\meas_{-1}$ now gives
\eqref{eq:kato-codim-one-zero}.

Fix $p\in\mathcal R$. Since a finite-dimensional $\RCD$ space is
essentially non-branching and satisfies the measure contraction property,
we can apply \cite[Proposition~A.6]{KapMonGT} to
$(X,\bar\dist,\bar\meas)$ and to each closed set $B_j$. In view of
\eqref{eq:kato-codim-one-zero}, for every $j\in\mathbb N$ there exists a
set $Q_j\subset X\setminus B_j$ of full $\bar\meas$-measure in
$X\setminus B_j$ such that, for every $q\in Q_j$, there is a
$\bar\dist$-minimizing geodesic from $p$ to $q$ entirely contained in $X\setminus B_j=M_j$.

We also use the fact that, for every fixed $p$ in a finite-dimensional
$\RCD$ space, the minimizing geodesic joining $p$ to $q$ is unique for
$\bar\meas$-a.e.\ $q\in X$; see
\cite[Lemma~4.1]{HanMondinoAngles}.
Let $Q_0\subset X$ be the corresponding full-measure set of points with
this uniqueness property. Since $\mathcal R$ has full $\bar\meas$-measure and
$\mathcal R\subset M_j$ for every $j$, the set
\[
    Q:=
    \mathcal R\cap Q_0\cap\bigcap_{j=1}^{\infty}Q_j
\]
has full $\bar\meas$-measure in $\mathcal R$. Fix $q\in Q$ and denote by
$\gamma_{p,q}$ the unique $\bar\dist$-minimizing geodesic joining $p$ to
$q$. For every $j$, \cite[Proposition~A.6]{KapMonGT} provides a minimizing geodesic from
$p$ to $q$ contained in $M_j$. By uniqueness, this geodesic must be
$\gamma_{p,q}$. Hence
\[
    \gamma_{p,q}\subset M_j
    \qquad\text{for every }j\in\mathbb N.
\]
Using \eqref{eq:kato-reg-intersection}, we conclude that
\[
    \gamma_{p,q}
    \subset
    \bigcap_{j=1}^{\infty}M_j
    =
    \mathcal R.
\]
Moreover, by \eqref{eq:kato-RCD-comparison},
\begin{align*}
    \operatorname{Length}_{\dist}(\gamma_{p,q})
    &\leq
    \operatorname{Length}_{\bar\dist}(\gamma_{p,q})\\
    &=
    \bar\dist(p,q)\\
    &\leq
    L\dist(p,q).
\end{align*}
Since $\bar\meas$ and $\haus^N$ are mutually absolutely continuous, the
same conclusion holds for $\haus^N$-a.e.\ $q\in\mathcal R$. Since the preceding argument applies to every $p\in\mathcal R$, Fubini's
theorem gives, for $\haus^N\times\haus^N$-a.e.\ pair
\[
    (p,q)\in
    (\mathcal R\cap B_1(x))
    \times
    (\mathcal R\cap B_1(x)),
\]
a rectifiable curve $\gamma\subset\mathcal R$ joining $p$ to $q$ and
satisfying
\[
    \operatorname{Length}_{\dist}(\gamma)
    \leq L\dist(p,q).
\]
Thus $\mathcal R$ is a.e.\ quasi-convex in the sense of
Definition~\ref{def-ae-quasiconvex}.
\end{proof}

\subsection{Local degree for (ETR) charts}\label{sec-localdegree}

We first establish that a Lang--Wenger current on a space induces a
local current, in the sense of Lang, on an open subset. This distinction will
be useful below, since an (ETR) chart is defined only on an open
subset.

Throughout the paper, $\Liploc(X)$ denotes
the functions that are Lipschitz on every bounded
subset of $X$, as in \cite{LangWenger}.
We write $\Lipl(X)$ for the functions that are
Lipschitz in a neighborhood of each point of $X$,
corresponding to $\operatorname{Lip}_{\rm loc}(X)$
in \cite{Lang}.
The two classes coincide when $X$ is proper, but
need not coincide on an open subset of a proper space. Moreover, we let
$\operatorname{Lip}_c(X)$ denote the Lipschitz
functions with compact support in $X$.

\begin{lemma}\label{lem:restriction to Lang}
Let $(X,\dist)$ be a metric space, $R\in\LWc{X}{n}$, and
$V\subset X$ be an open and locally compact subset. Then $R$ induces a 
current $R^V$ with locally finite mass on $V$, in the sense of Lang, given by 
\begin{equation}\label{eq:def-RV}
 R^V(f,\pi_1,\ldots,\pi_n)
 :=
 R(\widetilde f,\widetilde\pi_1,\ldots,\widetilde\pi_n)  
\end{equation}
for all
\[
 (f,\pi_1,\ldots,\pi_n)
 \in
 \operatorname{Lip}_c(V)\times
 [\Lipl(V)]^n,
\]
where for an open set $W$ such that
\[
 \spt(f)\subset W\Subset V,
\]
let $\widetilde f\in\Lipbs(X)$ be the extension of $f$ by zero outside
$V$, and, for every $i=1,\ldots,n$, let
$\widetilde\pi_i\in\Lip(X)$ be a Lipschitz extension of
$\pi_i|_{\overline W}$. 
The right-hand side of \eqref{eq:def-RV} is independent of all these choices. Moreover,
\begin{equation}\label{eq:mass-RV}
 \|R^V\|=\|R\|\rstr V.
\end{equation}
If in addition $R\in\LWnc{X}{n}$, then $\partial R^V=(\partial R)^V$. In particular, $
 \|\partial R\|(V)=0$ implies 
$\partial R^V=0$. If $R\in\LWic{X}{n}$, then $R^V$ is locally integral.
\end{lemma}

\begin{proof}
Since $\spt(f)$ is compact and contained in $V$, the extension of $f$ by
zero outside $V$ is a Lipschitz function on $X$ with bounded support.
Since each $\pi_i\in\Lipl(V)$ and $\overline W$ is
compact, $\pi_i|_{\overline W}$ is Lipschitz, and can be extended to a Lipschitz function $\tilde \pi_i$.
 
Let us check that \eqref{eq:def-RV} is well-defined. Suppose that
$\widetilde\pi_i$ and $\widetilde\pi_i'$ are two possible extensions.
They agree on a neighborhood of $\spt(f)$ and hence, by locality of
$R$, replacing $\widetilde\pi_i$ by $\widetilde\pi_i'$ does not change
the value of the current. The same argument shows that the definition is
independent of the choice of $W$. Multilinearity and locality follow from the corresponding properties
of $R$ and the independence of the extensions.

For continuity, suppose that
$(f_j,\pi_{j,1},\ldots,\pi_{j,n})$
converges to $(f,\pi_1,\ldots,\pi_n)$ in the
topology of \cite[Section~2]{Lang}.
Then there is a compact set $K\subset V$ containing the supports
of all $f_j$ and $f$. Choose an open set $W \subset V$ such that 
$K\subset W\Subset V$.
Since $\overline W$ is compact, 
for each $i$, the functions $\pi_{j,i}$ and
$\pi_i$ have a common Lipschitz bound $L_i$
on $\overline W$ and converge uniformly there.
Then their McShane extensions to $X$ have uniform Lipschitz constant $L_i$ and converge uniformly
on $X$. Writing $\widetilde f_j,\widetilde f$ for the zero extensions, we get
\[
 \left|
 R(\widetilde f_j-\widetilde f,\widetilde\pi_j)
 \right|
 \leq
\left(\prod_{i=1}^n L_i\right)
 \|f_j-f\|_\infty\,\|R\|(K)
 \to 0
\]
as $j \to \infty$.
Then, continuity of $R$ gives $R(\widetilde f,\widetilde\pi_j)
 \to
 R(\widetilde f,\widetilde\pi)$
as $j \to \infty$. Thus $R^V$ satisfies Lang's continuity axiom.

Let $U\Subset V$ be open. Given a test family of functions allowed to calculate 
$\mass_U(R^V)$ 
extend the first functions by zero and extend the
$1$-Lipschitz coordinate functions without increasing
their Lipschitz constants, this gives a 
family of functions allowed to calculate 
$\mass_U(R)$. Conversely, restrict a test family for $R$ to $V$. Its first functions have compact
support because their supports lie in $U\Subset V$.
The values of the two functionals agree, and hence
\[
 \mass_U(R^V)=\mass _U(R)=\|R\|(U).
\]
In particular, $R^V$ has locally finite mass. By regularity of the mass measures, equality on
relatively compact open subsets of $V$ implies
\eqref{eq:mass-RV}.

We next  prove $\partial R^V= (\partial R)^V$. Let
$(f,\pi)
 \in
\operatorname{Lip}_c(V)\times
 [\Lipl(V)]^{n-1}$,
and choose $\sigma\in\operatorname{Lip}_c(V)$ such that
$\sigma=1$ on a neighborhood of $\spt(f)$.
Choose an open set $W\Subset V$ containing $\spt(\sigma)$.
Let $\widetilde f$ and $\widetilde\sigma$ be the zero
extensions of $f$ and $\sigma$ to $X$, and let
$\widetilde\pi_i\in\Lip(X)$ extend
$\pi_i|_{\overline W}$.
Since $\widetilde f$ agrees with $f$ throughout $V$, it is also
an admissible extension when $f$ occupies a coordinate
slot in $R^V(\sigma,f,\pi)$.
Moreover, $\operatorname{spt}(\widetilde f)=\operatorname{spt}(f)$, so that $\widetilde\sigma=1$ on a neighborhood of
$\operatorname{spt}(\widetilde f)$.
Therefore,
\begin{align*}
 \partial R^V(f,\pi)
&=R^V(\sigma,f,\pi)=R(\widetilde\sigma,\widetilde f,\widetilde\pi)=\partial R(\widetilde f,\widetilde\pi)\\
 &=(\partial R)^V(f,\pi),
\end{align*}
which proves $ \partial R^V= (\partial R)^V$.

Since $R\in\LWnc{X}{n}$, we have
$\partial R\in\LWc{X}{n-1}$.
Applying \eqref{eq:mass-RV} to $\partial R$ gives
\[
 \|\partial R^V\|
 =\|(\partial R)^V\|
 =\|\partial R\|\rstr V.
\]
Thus both $R^V$ and $\partial R^V$ have locally finite
mass, so $R^V$ is locally normal in the sense of Lang.

It remains to verify the last assertion. Assume that
$R\in\LWic{X}{n}$. We show first that $R^V$ is locally integer
rectifiable. Since $R$ is locally integer rectifiable there exists a countably $n$-rectifiable Borel set $F\subset X$
such that $\|R\|(X\setminus F)=0$.
By \eqref{eq:mass-RV}, $\|R^V\|$ is concentrated on
$F\cap V$, which is countably $n$-rectifiable in $V$. 

Let $A\Subset V$ be a Borel set and $\varphi \in \Lip(V,\mathbb R^n)$. Choose $W$ such that $\overline A\subset W\Subset V$.
The restriction $\varphi|_{\overline W}$ is Lipschitz and therefore
admits a Lipschitz extension
$\widetilde\varphi$ to $X$.
By the definition of $R^V$ and locality,

\begin{equation}\label{eq:local-IR}
 \varphi_\#(R^V\rstr A)(f,\pi)
 =
 \widetilde\varphi_\#(R\rstr A)(f,\pi)
\end{equation}
for all $(f,\pi) \in \Lip_c(\mathbb R^n) \times \Lipl(\mathbb R^n)^n \subset \Lip_{bs}(\mathbb R^n) \times \Liploc(\mathbb R^n)^n$. 
Since $R$ is locally integer rectifiable, the current on the right-hand
side is of the form
\(
 \curr{\theta}
\)
for some $\theta\in L^1(\setR^n,\mathbb Z)$. Thus $R^V$ is locally integer
rectifiable.

By $ \partial R^V= (\partial R)^V$, the same argument applied to $\partial R$
shows that $\partial R^V$ is locally integer rectifiable. Therefore
$R^V$ is a locally integral current on $V$ in the sense of Lang.
\end{proof}

The following observation will allow us to pass from a local
pushforward to the compact biLipschitz pieces appearing in (ETR).

\begin{lemma}\label{lem:local-pushforward-bilip-piece}
Let $R\in\LWc{X}{n}$, $V\subset X$ be open and locally compact, and
$u:V\to u(V)$
be a Lipschitz homeomorphism onto its image.
Let $K\subset V$ be compact and assume
that $u|_K:K\to u(K)$
is biLipschitz. If $u_\#R^V=0$, then $R^V\rstr K=0$, and hence $\|R\|(K)=0$.
\end{lemma}

\begin{proof} 
Since $R^V$ has locally finite mass, its Borel restriction $Q:=R^V\rstr K$
is a current on $V$ with support contained in $K$.
Compactness of $K$ gives
\[
 \mathbf M(Q)=\|R^V\|(K)<\infty.
\]
Now, by \cite[Proposition~3.3]{Lang}, since $K$ is locally compact and the support of $Q$ is contained in $K$, there is a unique
current $Q_K$ on $K$ given by 
\[
\iota_\#Q_K=R^V\rstr K,
\]
where $\iota:K\hookrightarrow V$ is the inclusion map. Moreover,  $\mass(Q)= \mass(Q_K)$.

Using the restriction identity for pushforwards,
the injectivity of $u$, and functoriality, we obtain
\begin{align*}
 (u_\#R^V)\rstr u(K)
 &=u_\#\bigl(R^V\rstr u^{-1}(u(K))\bigr)\\
 &=u_\#(R^V\rstr K)\\
 &=u_\#\iota_\#Q_K\\
 &=j_\#(u|_K)_\#Q_K,
\end{align*}
where $j:u(K)\hookrightarrow u(V)$ is the inclusion map.

If $u_\#R^V=0$, then
$j_\#(u|_K)_\#Q_K=0$. 
The uniqueness assertion in
\cite[Proposition~3.3]{Lang}, applied to the zero
current on $u(V)$ and the compact subset $u(K)$,
implies that $(u|_K)_\#Q_K=0$.
Since $u|_K$ is biLipschitz, pushing forward by its
Lipschitz inverse gives
\[
 Q_K
 =\bigl((u|_K)^{-1}\bigr)_\#(u|_K)_\#Q_K
 =0.
\]
Consequently, $Q=\iota_\#Q_K=0$, and
$\|R\|(K)
 =\|R^V\|(K)
 =\mass(Q)
 =0$.
\end{proof}

We can now state the local Euclidean constancy result that will be used
throughout the sequel.

\begin{lemma}[Local Euclidean constancy]
\label{lem:local-euclidean-constancy}
Let $(X,\dist)$ be a metric space, $V\subset X$ be a connected open
and locally compact subset, and
\[
 u:V\to u(V)\subset\setR^N
\]
be a Lipschitz homeomorphism onto an open set. Let
$R\in\LWic{X}{N}$ and assume that
\begin{equation}\label{eq:zero-boundary-V}
 \|\partial R\|(V)=0.
\end{equation}
Then, as local currents on $u(V)$,
\begin{equation}\label{eq:local-euclidean-constancy}
 u_\#R^V=k[[u(V)]]
\end{equation}
for some $k\in\mathbb Z$.
\end{lemma}

\begin{proof}
By Lemma~\ref{lem:restriction to Lang}, $R^V$ is a locally integral
current on $V$ in the sense of Lang and, by
\eqref{eq:zero-boundary-V},
\(
 \partial R^V=0.
\)
Since $u$ is a homeomorphism onto $u(V)$, it is proper as a map from
$V$ to $u(V)$. Hence
\[
 S:=u_\#R^V
\]
is a locally integral current on $u(V)$ and
\[
 \partial S=u_\#(\partial R^V)=0.
\]

By \cite[Theorem~7.2]{Lang}, there exists $g\in {\rm BV}_{\rm loc}(u(V))$
such that
\[
 S=\curr{g}
 \qquad\text{and}\qquad
 \|\partial S\|=|Dg|.
\]
Since $\partial S=0$, we have $|Dg|=0$. The set $u(V)$ is connected,
because $V$ is connected and $u$ is a homeomorphism. It follows that
$g$ is constant $\Leb^N$-a.e. on $u(V)$. Finally, since $S$ is locally
integer rectifiable, this constant belongs to $\mathbb Z$. Hence $S=k[[u(V)]]$ for some $k\in\mathbb Z$.
\end{proof}

\begin{definition}[Local degree]\label{def:local-degree}
Under the assumptions of Lemma~\ref{lem:local-euclidean-constancy}, we
call the integer $k$ in
\eqref{eq:local-euclidean-constancy} the \emph{local degree of $u$ with
respect to $R$ on $V$}. We write
\[
 \deg_{\rm loc}(R,u,V):=k.
\]
Equivalently, we regard
$\deg_{\rm loc}(R,u,\,\cdot\,)$ as the constant integer-valued function
equal to $k$ on $u(V)$.
\end{definition}

We finally apply the preceding discussion to the charts given by
(ETR).

\begin{lemma}\label{lem:localdegreepmone}
There exists $\varepsilon_0(N)>0$ with the following property.
Let $(X,\dist,\haus^N)$ be a metric measure space and let $p\in X$
admit an (ETR) chart
\[
 u:B_r(p)\to u(B_r(p))\subset\setR^N
\]
with parameter $0<\varepsilon<\varepsilon_0(N)$.
Assume that $B_r(p)$ is connected.
Let $T\in\LWic{X}{N}$ and assume that
\begin{equation}\label{eq:local-orientation-assumptions}
 \|\partial T\|(B_r(p))=0,
 \qquad
 \|T\|\rstr B_r(p)=\haus^N\rstr B_r(p).
\end{equation}
Then
\[
 \deg_{\rm loc}(T,u,B_r(p))\in\{-1,1\}.
\]
\end{lemma}

\begin{proof}
Set
\[
 V:=B_r(p).
\]
By assumption, $V$ is connected. Moreover, $V$ is
locally compact because it is homeomorphic to the
open set $u(V)\subset\mathbb R^N$.
Therefore, Lemma~\ref{lem:local-euclidean-constancy}
gives an integer $k$ such that
\begin{equation}\label{eq:local-pushforward-k}
 u_\#T^V=k[[u(V)]]
\end{equation}
as local currents on $u(V)$. By Definition~\ref{def:local-degree},
$k=\deg_{\rm loc}(T,u,V)$.

We first prove that $k\neq0$. Let $\{K_i\}_{i\in\setN}$ be the
pairwise disjoint compact sets provided by the (ETR) chart, so that
\begin{equation}\label{eq:ETR-cover-degree}
 \haus^N\left(
 V\setminus\bigcup_{i\in\setN}K_i
 \right)=0
\end{equation}
and $u|_{K_i}$ is biLipschitz for every $i$.

Suppose by contradiction that $k=0$. Then $u_\#T^V=0$, and 
by Lemma~\ref{lem:local-pushforward-bilip-piece}, for every
$i\in\setN$ we have
$\haus^N(K_i)= \|T\|(K_i)
 = 0$. Together with \eqref{eq:ETR-cover-degree}, this gives
$\haus^N(V)=0$. This contradicts the volume estimate in Definition~\ref{def:assumptions},
which gives
\[
 \haus^N(V)
 \geq
 (1-\varepsilon)\omega_Nr^N>0.
\]
Hence  $k\neq0$. It remains to estimate $|k|$.

By
\eqref{eq:local-pushforward-k} and the mass inequality for
pushforwards,
\begin{align*}
 |k|\,\Leb^N(u(V))
 &=
 \|u_\#T^V\|(u(V))\\
 &\leq
 \Lip(u)^N\,\|T^V\|(V)\\
 &=
 \Lip(u)^N\,\haus^N(V).
\end{align*}
The properties of an (ETR) chart give
\[
 B_{(1-2\varepsilon)r}^{\setR^N}(0)
 \subset u(V),
 \qquad
 \Lip(u)\leq1+\varepsilon,
\]
and
\[
 \haus^N(V)
 \leq
 (1+\varepsilon)\omega_Nr^N.
\]
Therefore
\begin{align*}
 |k|\,\omega_N(1-2\varepsilon)^Nr^N
 &\leq
 |k|\,\Leb^N(u(V))\\
 &\leq
 (1+\varepsilon)^N\haus^N(V)\\
 &\leq
 (1+\varepsilon)^{N+1}\omega_Nr^N.
\end{align*}
It follows that
\begin{equation}\label{eq:k-bound-ETR}
 |k|
 \leq
 \frac{(1+\varepsilon)^{N+1}}
 {(1-2\varepsilon)^N}.
\end{equation}
Choose $\varepsilon_0(N)>0$ sufficiently small such that
\[
 \frac{(1+\varepsilon)^{N+1}}
 {(1-2\varepsilon)^N}<2
\]
for every $0<\varepsilon<\varepsilon_0(N)$. Since $k \neq 0$ is an integer, \eqref{eq:k-bound-ETR} implies $|k|=1$.
This concludes the proof.
\end{proof}

\begin{lemma}[Compatibility of local and localized degrees]
\label{lem:local-to-global-degree}
Let $X$ be a complete metric space, $V\subset X$
be open and connected, and
\[
 u:V\to u(V)\subset\mathbb R^n
\]
be a Lipschitz homeomorphism onto an open set.
Let $R\in\LWic{X}{n}$ satisfy $\|\partial R\|(V)=0$,
and write
\[
 u_\#R^V=k[[u(V)]]
\]
as in Lemma~\ref{lem:local-euclidean-constancy}.
Let $U\Subset V$ be open and suppose that
$R\rstr U\in\bI_n(X)$.
Then every Lipschitz extension
$\widehat u:X\to\mathbb R^n$ of $u$ satisfies
\[
 \deg(R\rstr U,\widehat u,a)=k
 \qquad\text{for $\mathcal L^n$-a.e. }a\in u(U).
\]
\end{lemma}

\begin{proof}
The set $V$ is locally compact because it is homeomorphic
to the open set $u(V)\subset\mathbb R^n$.
Hence, we can apply Lemma 
~\ref{lem:local-euclidean-constancy} to get $u_\#R^V=k[[u(V)]]$. 

Let $f\in\operatorname{Lip}_c(u(V))$ such that 
$\spt(f) \subset u(U)$, and denote its zero
extension to $\mathbb R^n$ by $\widetilde f \in  \Lipbs(\mathbb R^n)$. 
Since $u$ is an homeomorphism, $f\circ u$ belongs to $\operatorname{Lip}_c(V)$,
$K:=\spt(f\circ u)=u^{-1}(\spt(f)) \subset U$ and 
$(\widetilde f \circ \hat u)|_{K}= (f \circ u)|_K$.
Consequently, taking 
$\pi_1,\ldots,\pi_n\in
\Lip(\mathbb R^n)\subset \Liploc(\mathbb R^n)$,
we get
\begin{align*}
 \widehat u_\#(R\rstr U)(\widetilde f,\pi)
 &=R\bigl(1_U(\widetilde f\circ\widehat u),
                    \pi\circ\widehat u\bigr)\\
 &=R^V(f \circ u ,\pi\circ u)\\
 &=(u_\#R^V)( f,\pi|_{u(V)})\\
 &=k[[u(V)]](f,\pi|_{u(V)}).
\end{align*}
Taking $\pi=\operatorname{id}_{\mathbb R^n}$ gives
\[
 \int_{u(U)}
 f(a)\,\deg(R\rstr U,\widehat u,a)\,d\mathcal L^n(a)
 =
 k\int_{u(U)}f(a)\,d\mathcal L^n(a).
\]
Since this holds for every
$f\in\operatorname{Lip}_c(u(V))$ with $\spt(f)\subset u(U)$, the almost-everywhere equality follows.
\end{proof}

Next, we establish a useful consequence of the (LDB) condition. 

\begin{lemma}\label{lem:set}
Let $(X,\dist,\haus^N)$  be a complete  metric measure
space satisfying 
(LDB).  Assume there exists  $T \in 
\LWic{X}{N}$ with $\|T\|=\haus^N$. Then $\set(T)=X$.
\end{lemma}

\begin{proof}
If
\(
\|T\|=\mathcal H^N
\)
and (LDB) holds, then for every \(x\in X\)
\[
\liminf_{r\downarrow0}
\frac{\|T\|(B_r(x))}{r^N}
=
\liminf_{r\downarrow0}
\frac{\mathcal H^N(B_r(x))}{r^N}>0.
\]
By the definition of \(\operatorname{set}(T)\), \(x\in\operatorname{set}(T)\). Hence $\set(T)=X$.
\end{proof}

\section{Proofs of Main Results}\label{sec-ProofsPropCor}

\subsection{Local rigidity and uniqueness of orientations}\label{SubSec:Pf1213}

We now prove Proposition~\ref{prop:locrep} and Corollary~\ref{cor-2orientations}.    We formulate the proof in the curvature-free (ETR) setting.

\begin{lemma}\label{lem:local-relative-rigidity}
Let $(X,\dist,\haus^N)$ satisfy {\rm(ETR)}, let
$T\in\LWic{X}{N}$ orient $X$, and let $S\in\LWic{X}{N}$.
Suppose that $p\in\mathcal R$ and that $\partial S=0$ in a
neighborhood of $p$. Then there exist $r_p>0$ and $k_p\in\mathbb Z$
such that
\begin{equation}\label{eq-local-relative-rigidity}
    S\rstr B_{r_p}(p)
    =
    k_p\,T\rstr B_{r_p}(p).
\end{equation}
\end{lemma}

\begin{proof}
Choose an (ETR) chart $u:V=B_r(p)\to u(V)$ with
parameter $\varepsilon<\varepsilon_0(N)$ and with
$V$ contained in the neighborhood where $\partial S=0$.
By Lemmas~\ref{lem:local-euclidean-constancy}
and~\ref{lem:localdegreepmone},
\[
 u_\#S^V=\ell[[u(V)]],
 \qquad
 u_\#T^V=\sigma[[u(V)]],
 \qquad
 \ell\in\mathbb Z,\quad \sigma\in\{-1,1\}.
\]
Set $k_p:=\sigma\ell$ and $D:=S-k_pT$.
Then $u_\#D^V=0$.
For every compact biLipschitz piece $K_i$ of this
same chart, Lemma~\ref{lem:local-pushforward-bilip-piece}
gives $\|D\|(K_i)=0$.
Since these sets cover $V$ up to a
$\haus^N$-negligible set, 
and
$\|D\|\ll\haus^N$ by \eqref{eq-TmeasRep},
we obtain $\|D\|(V)=0$.
Since $V$ is Borel, by \eqref{eq-massRestriction}, $\|D\|\llcorner V= \|D \llcorner V \|$. Thus, 
$D \llcorner V =0$.
Taking $r_p=r$ proves \eqref{eq-local-relative-rigidity}.
\end{proof}

\begin{proposition}\label{prop:locrep1}
Let $(X,\dist,\haus^N)$ satisfy (ETR) and assume that its regular set is a.e. quasi-convex.  Let $T\in\LWic{X}{N}$ orient $X$.  If $S\in\LWic{X}{N}$ satisfies $\partial S=0$ on $X$, then there exists $k\in\mathbb Z$ such that $S=kT$.
\end{proposition}

\begin{proof}
By Lemma~\ref{lem:local-relative-rigidity}, the relative integer coefficient of $S$ with respect to $T$ is locally constant around every regular point.  We first show that it is essentially constant on a small ball centered at an arbitrary point $y\in X$.  Let $C_y$ be given by Definition~\ref{def-ae-quasiconvex} and choose $0<\rho<1$.  For $\haus^N\times\haus^N$-a.e. $p,q\in\mathcal R\cap B_\rho(y)$ there is a rectifiable curve $\gamma\subset\mathcal R$ joining $p$ to $q$.  Covering the compact image of $\gamma$ by finitely many balls to which Lemma~\ref{lem:local-relative-rigidity} applies, and taking a chain of overlapping members of this cover, shows that the local coefficients at $p$ and $q$ coincide.

Indeed, on a non-empty open overlap two relations $S=k_iT$ and $S=k_{i+1}T$ imply $(k_i-k_{i+1})T=0$, while $X=\set(T)$ implies that every non-empty open set has positive $\|T\|=\haus^N$ measure.  By Fubini, the coefficient is therefore constant for $\haus^N$-a.e. point of $B_\rho(y)$.

Thus every point $y\in X$ has a neighborhood on which $S=k_yT$ for some integer $k_y$.  Since every metric ball is connected by (ETR), $X$ is connected.  The coefficients agree on overlapping members of the above open cover, and hence all of them agree, proving the claim.
\end{proof}

\begin{proposition}\label{prop:locrep2}
Let $(X,\dist,\haus^N)$ satisfy (ETR) and assume that its regular set $\mathcal R$ is a.e.\ quasi-convex.  Let $T\in\LWic{X}{N}$ orient $X$, and let $S\in\LWic{X}{N}$.  If $x\in X$ and $s>0$ satisfy
\[
 \|\partial S\|(B_s(x))=0,
\]
then there exists $k\in\mathbb Z$ such that $S\rstr B_s(x)=k\,T\rstr B_s(x)$.
\end{proposition}

\begin{proof}
The only issue is to make sure that the curves used to propagate the local coefficient remain inside $B_s(x)$, where the hypothesis on the mass of the boundary is available.  If $B_s(x)=X$, the conclusion follows directly from Proposition~\ref{prop:locrep1}, so we may assume that $X\setminus B_s(x)\neq\varnothing$.  Fix $y\in B_s(x)$ and let $C_y$ be the constant in Definition~\ref{def-ae-quasiconvex}.  Choose
\begin{equation}\label{eq-rho-prop44}
 0<\rho_y<\min\left\{1,\frac{\dist(y,X\setminus B_s(x))}{1+2C_y}\right\}.
\end{equation}
For $\haus^N\times\haus^N$-a.e. pair $p,q\in\mathcal R\cap B_{\rho_y}(y)$, Definition~\ref{def-ae-quasiconvex} gives a rectifiable curve $\gamma\subset\mathcal R$ joining $p$ to $q$ and satisfying
\[
 \operatorname{Length}(\gamma)\le C_y\dist(p,q).
\]
By \eqref{eq:curve-Ball} we know that that 
$\gamma([0,1])\subset B_{(1+2C_y)\rho_y}(y)$. 
Hence, by \eqref{eq-rho-prop44}, $\gamma([0,1])\subset B_s(x)$.  

We can therefore cover $\gamma([0,1])$ by finitely many regular balls contained in $B_s(x)$ on which Lemma~\ref{lem:local-relative-rigidity} applies.  Passing along a chain of overlapping balls shows that the relative coefficients at $p$ and $q$ agree.  By Fubini, there exists $k_y\in\mathbb Z$ such that
\[
 S\rstr B_{\rho_y}(y)=k_yT\rstr B_{\rho_y}(y).
\]
Here we use again that $\mathcal R$ has full $\haus^N$ measure and that the mass measures of top-dimensional rectifiable currents are absolutely continuous with respect to $\haus^N$.

The balls $B_{\rho_y}(y)$, $y\in B_s(x)$, form an open cover of $B_s(x)$.  If two of them intersect, the intersection is a non-empty open set and hence has positive $\|T\|=\haus^N$ measure; consequently their coefficients coincide.  Since $B_s(x)$ is connected by (ETR), the intersection graph of the cover is connected.  Thus all $k_y$ coincide with a single integer $k$. This concludes the proof. 
\end{proof}

\begin{proof}[Proof of Proposition~\ref{prop:locrep}]
By Proposition~\ref{prop-ex-assms}, a non-collapsed $\RCD(K,N)$ space without boundary satisfies (ETR) and has an a.e.\ quasi-convex regular set.  Proposition~\ref{prop:locrep2} therefore applies and gives the conclusion.
\end{proof}

\begin{proof}[Proof of Corollary~\ref{cor-2orientations}]
Let $S$ be another current orienting $(X,\dist,\haus^N)$.  Proposition~\ref{prop:locrep1} gives $S=kT$ for some $k\in\mathbb Z$.  Since
\[
 \|S\|=\|T\|=\haus^N,
\]
we have $|k|=1$.  Thus $S=T$ or $S=-T$.
\end{proof}

\subsection{Proof of Theorem \ref{thm-main} and Theorem \ref{thm-main2}}\label{sec-ProofsThms}

The next result is inspired by {\cite[Theorem 4.1, i)]{MatveevPortegies17}} and will be used to prove Theorem \ref{thm-main} at the end of the section.

\begin{theorem}\label{thm-mainD}
Let $(X_i,\dist_i,\haus^N)$ be a sequence of  non collapsed $\RCD(K,N)$ spaces,  for some $K \in \setR$ and $N \in \mathbb N$, without boundary. 
Assume that
\begin{itemize}
\item $\sup_i {\rm diam}(X_i)<\infty$;
\item  $(X_i, \dist_i)\to (X,\dist)$ in GH sense; 
\item Each $X_i$ is oriented by an $N$-dimensional integral current $T_i \in \bI_N(X_i), \partial T_i=0, \ \|T_i\|=\haus^N$.
\end{itemize}
Then, either one of the following happens:
\begin{itemize} 
\item \emph{Collapsing case}, i.e.
$\lim_{i \to \infty}\haus^N(X_i) =0$: 
the sequence $T_i$  converges in the flat topology to the zero current $T=0 \in \bI_N(X)$;
\item \emph{Non-collapsing case}, i.e.
$\lim_{i \to \infty}\haus^N(X_i) >0$: the limit  $(X,\dist,\haus^N)$
is a non-collapsed $\RCD(K,N)$ space without boundary,
the currents $T_i$ subconverge in the flat topology to a non zero current $T \in \bI_N(X)$, with $\partial T=0$, $X=\set(T)$, $\|T\|=\haus^N$; i.e., $T$ is an orientation in the sense of currents for $(X,\dist,\haus^N)$. 
\end{itemize}
\end{theorem}

\begin{proof}
Given that  $||T_i||=  \haus^N$ and that $\sup_i {\rm diam}(X_i)<\infty$, we can apply Bishop-Gromov's volume comparison (see \cite[Thm\,2.3]{Sturm:II})) to infer that
$$\sup_{i \in \mathbb N} \Big \{ ||  T_i || (X_i) \Big \}  < \infty.$$
Thus, Wenger's compactness theorem \ref{thm-WengerCpctness} implies that there exists a subsequence of the $(X_i,\dist_i, T_i)$, that we do not relabel, that converges in intrinsic flat sense to an $N$-dimensional integral current space 
$(Y,\dist_Y,T)$. 

If $\lim_{i \to \infty}\haus^N(X_i) =0$,
since $||T_i||=  \haus^N$ and by the lower semi-continuity of the mass under flat convergence, we obtain that $T=0$ and thus  
$(Y,\dist_Y,T)$ is the zero integral current space.

Otherwise, the GH-continuity of the $\haus^N$-measure of metric balls  under lower Ricci bounds (also known as Colding's volume convergence, see \cite[Thm.\ 1.3]{DPG}, after \cite{Colding:VolumeConv, CC:97}), we have that $$\lim_{i \to \infty}\haus^N(X_i)=\haus^N(X)>0$$ and the sequence $(X_i, \mathsf d_i, \haus^N)$ converges in mGH sense to $(X, \dist, \haus^N)$. By the stability of the $\RCD$ condition under mGH convergence (see for instance \cite{GMS15}), it follows that $(X, \dist, \haus^N)$ is an $\RCD(K,N)$ space.
Moreover,  since $\partial X_i=\emptyset$ then the limit space $X$ has no boundary, thanks to \cite[Thm.\ 5.1]{KapMonGT}.
We will show that $T$ is an orientation for 
$(X,\dist,\haus^N)$.

By Theorem
\ref{thm-IFsubsetGH},  there exist a complete metric space $W$, and isometric embeddings 
\[ 
\varphi_i: X_i \to W \qquad \forall \, i \in \mathbb N,
\]
and $\varphi: X \to W$, 
such that 
\begin{eqnarray}
       & \varphi_{i \sharp} T_i \to \varphi_\sharp T  \quad \text{ in flat sense} \label{eq:phiiTitophiT}\\
     &   \varphi_{i}(X_i) \to \varphi(X) \quad \text{in Hausdorff sense in $W$} \\
     &\varphi(Y)=\set(\varphi_\sharp(T)) \subset \varphi(X).
\end{eqnarray}

 By the definition of flat convergence, $ \partial (\varphi_{i \sharp} T_i) =0$ converges to $\partial (\varphi_\sharp T)$. 
Hence, $\partial (\varphi_\sharp T)=0$.   
Since the pushforward and boundary operators commute, we have $\partial T=0$.

We prove that $\|T\|=\haus^N$ by working at an arbitrary regular point
$p\in\mathcal R\subset X$. Since $\varphi_i(X_i)\to\varphi(X)$
in Hausdorff distance in $W$, we may choose $p_i\in X_i$ such that
\[
 z_i:=\varphi_i(p_i)\to
    z:=\varphi(p)
    \qquad\text{in }W.
\] 
 
Fix
\[
    0<\varepsilon<
    \min\{\varepsilon_0(N),1/10\},
\]
where $\varepsilon_0(N)$ is given by Lemma
\ref{lem:localdegreepmone}, and choose 
\begin{equation}\label{eq-def-c-eta}
0<\eta<\varepsilon, \qquad  c_\eta:=(1-\eta)2^{-\eta}\geq1-2\varepsilon,
\end{equation}
with $\eta$
sufficiently small such that
the almost-splitting maps provided by
Theorem~\ref{thm:biholder} have splitting error at most 
the bound that appears in Proposition~\ref{prop:BilipPartition}. Note that such $c_\eta$ exists since $c_\eta\to1$ as $\eta\downarrow0$.

Since $p$ is $N$-regular, there exists
$r_\varepsilon(p)>0$ such that, for every
$0<r<r_\varepsilon(p)$, setting
\[
    R_0:=2r,
\]
the rescaling by $R_0^{-1}$ satisfies the geometric
hypotheses of Theorem~\ref{thm:biholder} and
Proposition~\ref{prop:BilipPartition}, with enough
room in the smallness thresholds to apply the same
results to the approximating spaces.
Indeed, the rescaled balls at $p$ converge to Euclidean
balls as $r\downarrow0$, and the rescaled curvature
lower bound tends to zero.

Fix any such $r$.
Pointed measured Gromov--Hausdorff convergence implies
that, for all sufficiently large $i$, the corresponding
rescaled balls centered at $p_i$ satisfy the same
hypotheses.
All subsequent subsequence extractions may depend
on $p$, $\varepsilon$ and $r$; the current limit $T$
remains unchanged.

Applying Theorem~\ref{thm:biholder}, we choose maps
\[
    u_i:B_{8R_0}(p_i)\to\setR^N,
    \qquad
    u_i(p_i)=0,
\]
with splitting error $\eta$.
Now we choose a splitting map on the limit space. Indeed, the maps $u_i$ are harmonic componentwise and satisfy a uniform
local Lipschitz bound. Hence, after passing to a further subsequence,
the compactness argument for splitting maps under pmGH convergence,
as in the proof of
\cite[Theorem~3.8(ii)]{BrueNaberSemola20}, gives a map
\[
    u:B_{4R_0}(p)\to\setR^N,
    \qquad
    u(p)=0,
\]
such that
\begin{equation}\label{eq:compatible-splitting-convergence}
    u_i\to u
\end{equation}
locally uniformly in the pmGH sense. More precisely, in the common
realization $W$, for every $s<4R_0$ the graphs of the component
functions
\[
    u_i^\alpha\circ\varphi_i^{-1}:
    \varphi_i(\overline B_s(p_i))\to\setR,
    \qquad \alpha=1,\ldots,N,
\]
converge in Hausdorff distance in $W\times\setR$ 
endowed with the maximum metric
to the graph of
\[
    u^\alpha\circ\varphi^{-1}:
    \varphi(\overline B_s(p))\to\setR.
\]
The stability of harmonic functions, gradients and Hessians integral bounds under pmGH
convergence shows that $u$ is again an almost-splitting map; equivalently,
this follows from the same compactness argument used in the proof of
\cite[Theorem~3.8(ii)]{BrueNaberSemola20}. 

By choosing $\eta$ sufficiently
small in terms of $\varepsilon$ and $N$, its splitting error is smaller
than $\varepsilon$. Moreover, the quantitative estimates in
Theorem~\ref{thm:biholder} pass to the limit. In particular, on the
smaller ball on which we work,
$|u(x)-u(y)|
    \leq (1+\varepsilon)\dist(x,y)$,
and the lower H\"older estimate implies that $u$ is one-to-one.
The Euclidean ball inclusion in
Theorem~\ref{thm:biholder}(iv) also passes to the limit.
Thus the coordinates $u_i$ and $u$ are now chosen compatibly.

 By \eqref{eq:phiiTitophiT} we can applying Lemma~\ref{lem:localization-flat-balls} to
\[
    S_i:=\varphi_{i\sharp}T_i,
    \qquad
    S:=\varphi_\sharp T,
    \qquad
    z_i:=\varphi_i(p_i),
    \qquad
    z:=\varphi(p).
\]
Hence, after passing to a further subsequence, and choosing 
\begin{equation}\label{eq:RforballLemma}
    R\in(r,3r/2)
\end{equation}
in the full-measure set given
in the lemma, where $r$ was fixed above and is given by $R_0=2r$,
we get 
\begin{equation}\label{eq:flat-localized-mainD}
    \dist_F^W
    \left(
        \varphi_{i\sharp}T_i^{\rm loc},
        \varphi_\sharp T^{\rm loc}
    \right)
    \to0.
\end{equation}
Here,
\[
    T_i^{\rm loc}:=T_i\rstr B_R(p_i),
    \qquad
    T^{\rm loc}:=T\rstr B_R(p),
\]
and note that since the embeddings are isometric,
we have 
\[
    \varphi_{i\sharp}T_i^{\rm loc}
    =
    (\varphi_{i\sharp}T_i)\rstr B_R(z_i),
\]
and
\[
    \varphi_\sharp T^{\rm loc}
    =
    (\varphi_\sharp T)\rstr B_R(z).
\]

We now construct extensions for the application of
Lemma~\ref{le:L1DegConv}. Set
\[
 D_i:=\varphi_i(\overline B_R(p_i)),
 \qquad
 D:=\varphi(\overline B_R(p)).
\]
For $\alpha=1,\ldots,N$, define
\[
f_i^\alpha:=u_i^\alpha\circ\varphi_i^{-1}:D_i\to\setR,
 \qquad
f^\alpha:=u^\alpha\circ\varphi^{-1}:D\to\setR.
\]
Let $L=1+\varepsilon$, a common Lipschitz bound for these
scalar functions, and take their canonical McShane extensions:
\begin{equation}\label{eq:canonical-McShane-ui}
 \widehat u_i^\alpha(w)
 :=\inf_{x\in D_i}
 \{f_i^\alpha(x)+L\dist_W(w,x)\},
 \qquad w\in W,
\end{equation}
\begin{equation}\label{eq:canonical-McShane-u}
 \widehat u^\alpha(w)
 :=\inf_{x\in D}
 \{f^\alpha(x)+L\dist_W(w,x)\},
 \qquad w\in W.
\end{equation}
They are $L$-Lipschitz and agree with the original
functions on their respective domains.

Since the sequence of compact $D_i$ converge in Hausdorff distance to $D$ in $W$, and the way we chose $u$, we know that 
\begin{equation}\label{eq:graph-convergence-splitting}
 \eta_i^\alpha:=
 \dist_H^{W\times\setR}
 \bigl(\operatorname{graph}(f_i^\alpha),
       \operatorname{graph}(f^\alpha)\bigr)
 \to0.
\end{equation}
Comparing the infima using
\eqref{eq:graph-convergence-splitting} gives
\begin{equation}\label{eq:McShane-uniform-estimate}
 \|\widehat u_i^\alpha-\widehat u^\alpha\|_{L^\infty(W)}
 \leq(1+L)\eta_i^\alpha
 \to0.
\end{equation}
Thus the maps
$\widehat u_i=(\widehat u_i^1,\ldots,\widehat u_i^N)$
and $\widehat u=(\widehat u^1,\ldots,\widehat u^N)$
have uniformly Lipschitz components and satisfy
\begin{equation}\label{eq:compatible-extensions-convergence}
 \sum_{\alpha=1}^N
 \|\widehat u_i^\alpha-\widehat u^\alpha\|_{L^\infty(W)}
 \to0.
\end{equation}
We may now apply Lemma~\ref{le:L1DegConv} with $S_i:=\varphi_{i\sharp}T_i^{\rm loc}$, $
    S:=\varphi_\sharp T^{\rm loc}$, 
    $\Phi_i:=\widehat u_i$ and 
    $\Phi:=\widehat u$. Since the first term in the estimate of
Lemma~\ref{le:L1DegConv} tends to zero by
\eqref{eq:flat-localized-mainD}, while the second tends to zero by
\eqref{eq:compatible-extensions-convergence}, we obtain
\begin{equation}\label{eq:degree-L1-new}
\deg(\varphi_{i\sharp}T_i^{\rm loc},\widehat u_i,\cdot)
 \to
 \deg(\varphi_\sharp T^{\rm loc},\widehat u,\cdot)
 \qquad\text{in }L^1(\setR^N).
\end{equation}

We now identify a common domain on which the degrees
have absolute value one and obtain a mass estimate
on the ball of radius $r$.

The rescaled lower H\"older estimate in
Theorem~\ref{thm:biholder}(iii) gives
\begin{equation}\label{eq:radial-chart-lower-bound}
    |u_i(q)|
    \geq
    (1-\eta)R_0^{-\eta}
    \dist_i(q,p_i)^{1+\eta}
    \qquad
    \text{for }q\in B_{R_0}(p_i).
\end{equation}
The same estimate holds for $u$, by the compatible
convergence of the maps.
Moreover,
\[
    B_{(1-2\eta)R_0}^{\setR^N}(0)
    \subset u_i(B_{R_0}(p_i)).
\]
Set
\begin{equation}\label{eq:common-degree-ball}
    E:=B_{c_\eta r}^{\setR^N}(0).
\end{equation}
Since $\eta<1/10$, $R_0=2r$ and \eqref{eq-def-c-eta}, we have
$c_\eta r<(1-2\eta)R_0$.
Thus every $a\in E$ has a preimage under $u_i$ in
$B_{R_0}(p_i)$. By
\eqref{eq:radial-chart-lower-bound}, any such preimage
must belong to $B_r(p_i)$: otherwise its image would
have norm at least $(1-\eta)R_0^{-\eta}r^{1+\eta}
    =c_\eta r$.
Consequently,
\[
    E\subset u_i(B_r(p_i)).
\]
We also have
\[
    E\subset u(B_r(p)).
\]
Indeed, fix $a\in E$ and choose $q_i\in B_r(p_i)$
with $u_i(q_i)=a$.
The common realization and properness allow us to pass
to a subsequence with $q_i\to q\in\overline B_r(p)$.
Convergence of the maps gives $u(q)=a$, while
\eqref{eq:radial-chart-lower-bound} and
$|a|<c_\eta r$ imply $\dist(q,p)<r$.

By Proposition~\ref{prop:BilipPartition}, each $u_i$
has the required full-measure biLipschitz decomposition
on $B_{R_0}(p_i)$.
Together with the other conclusions of
Theorem~\ref{thm:biholder}, this makes
$u_i|_{B_{R_0}(p_i)}$ an (ETR) chart with 
$\eta$ as splitting error.
Since $T_i$ is an orientation,
Lemma~\ref{lem:localdegreepmone} gives
\[
 \deg_{\rm loc}(T_i,u_i,B_{R_0}(p_i))
 =\sigma_i\in\{-1,1\}.
\]

Recall that $r<R<3r/2<R_0=2r$.
Since $X_i$ is proper,
$B_R(p_i)\Subset B_{R_0}(p_i)$,
and the choice of $R$ ensures that
$T_i^{\rm loc}=T_i\rstr B_R(p_i)\in\bI_N(X_i)$.
Choose a Lipschitz extension
$\widetilde u_i:X_i\to\mathbb R^N$ of
$u_i|_{B_{R_0}(p_i)}$.
Lemma~\ref{lem:local-to-global-degree}, applied with
$V=B_{R_0}(p_i)$ and $U=B_R(p_i)$,
gives
\[
 \deg(T_i^{\rm loc},\widetilde u_i,\cdot)
 =\sigma_i
 \qquad\text{a.e. on }u_i(B_R(p_i)).
\]
The maps $\widehat u_i\circ\varphi_i$ and
$\widetilde u_i$ agree on $\overline B_R(p_i)$,
which contains $\operatorname{spt}(T_i^{\rm loc})$.
Hence, by locality and functoriality of pushforwards,
\[
 (\widehat u_i)_\#
       (\varphi_{i\sharp}T_i^{\rm loc})
 =
 (\widetilde u_i)_\#T_i^{\rm loc}.
\]
Consequently,
\[
 g_i:=
 \deg(\varphi_{i\sharp}T_i^{\rm loc},
          \widehat u_i,\cdot)
 =\sigma_i
 \qquad\text{a.e. on }u_i(B_R(p_i)).
\]
Write
\[
    g:=
    \deg(\varphi_\sharp T^{\rm loc},
         \widehat u,\cdot).
\]
By \eqref{eq:degree-L1-new}, $g_i\to g$ in
$L^1(\setR^N)$, and
since
$E \subset u_i(B_r(p_i)) \cap u(B_r(p))$, 
$(g_i|_E)_i$ is Cauchy in $L^1(E)$. Since
\[
    \|g_i-g_j\|_{L^1(E)}
    =
    |\sigma_i-\sigma_j|\,\Leb^N(E),
\]
and $\Leb^N(E)>0$, the signs $\sigma_i$ are eventually
constant. Hence
\begin{equation}\label{eq:limit-degree-on-E}
    g=\sigma
    \qquad\text{a.e. on }E
\end{equation}
for some $\sigma\in\{-1,1\}$.

The restriction identity for pushforwards gives
\[
    \bigl(\widehat u_\#
          (\varphi_\sharp T^{\rm loc})\bigr)\rstr E
    =
    \widehat u_\#
    \bigl((\varphi_\sharp T^{\rm loc})
          \rstr\widehat u^{-1}(E)\bigr).
\]
Set $A:=B_R(p)\cap u^{-1}(E)$.
Since $u$ is injective on $B_{R_0}(p)$ and
$E\subset u(B_r(p))$, we have $A\subset B_r(p)$.
Moreover,
\[
\varphi_\sharp(T\rstr A)=(\varphi_\sharp T^{\rm loc})
       \rstr\widehat u^{-1}(E)
\]
is supported in $D:=\varphi(\overline B_R(p))$.
On $D$, the map $\widehat u$ agrees with
$u\circ\varphi^{-1}$ and has Lipschitz constant at most
$1+\varepsilon$.
Applying the pushforward mass inequality to $\varphi_\sharp(T\rstr A)$,
regarded as a current on $D$, and using
\eqref{eq:limit-degree-on-E}, we obtain
\[
    \Leb^N(E)
    \leq
    (1+\varepsilon)^N\|T\|(A)
    \leq
    (1+\varepsilon)^N\|T\|(B_r(p)).
\]
Thus
\begin{equation}\label{eq:lower-density-new}
    \frac{\|T\|(B_r(p))}{\omega_Nr^N}
    \geq
    \frac{c_\eta^N}{(1+\varepsilon)^N}
    \geq
    \frac{(1-2\varepsilon)^N}{(1+\varepsilon)^N}.
\end{equation}

For the fixed $\varepsilon$, the radius
$r\in(0,r_\varepsilon(p))$ was arbitrary.
Although the coordinate maps and subsequences may
depend on $r$, inequality
\eqref{eq:lower-density-new} concerns the same fixed
current $T$.
It therefore holds for every sufficiently small $r$.
Taking the lower limit as $r\downarrow0$ and then
letting $\varepsilon\downarrow0$ gives
\[
    \liminf_{r\downarrow0}
    \frac{\|T\|(B_r(p))}{\omega_Nr^N}
    \geq1.
\]

For the reverse inequality, volume convergence in
the common realization gives
\[
    \varphi_{i\sharp}\haus^N
    \rightharpoonup
    \varphi_\sharp\haus^N.
\]
Since $\|T_i\|=\haus^N$, lower semicontinuity of mass
under weak convergence of currents implies
\[
    \|\varphi_\sharp T\|
    \leq\varphi_\sharp\haus^N,
    \qquad\text{hence}\qquad
    \|T\|\leq\haus^N.
\]
At the $N$-regular point $p$,
\[
    \lim_{r\downarrow0}
    \frac{\haus^N(B_r(p))}{\omega_Nr^N}=1.
\]
Consequently,
\[
    \limsup_{r\downarrow0}
    \frac{\|T\|(B_r(p))}{\omega_Nr^N}
    \leq1.
\]
Combining the two inequalities yields
\[
    \lim_{r\downarrow0}
    \frac{\|T\|(B_r(p))}{\omega_Nr^N}=1.
\]

By the mass representation \eqref{eq-TmeasRep}, $\|T\|\ll\haus^N$.  Since $\haus^N(X\setminus\mathcal R)=0$ and the density is one $\haus^N$-a.e. on $\mathcal R$, differentiation of Radon measures gives
\[
 \|T\|=\haus^N.
\]
Since $(X,\dist, \haus^N)$ is a non-collapsed $\RCD(K,N)$ space without boundary, the combination of Proposition \ref{prop-ex-assms} and  Lemma \ref{lem:set} yields  that $\set(T)=X$.  Thus $T$ is an orientation of $(X,\dist,\haus^N)$.
\end{proof}

We finally prove Theorem \ref{thm-main}.  

\begin{proof}[Proof of Theorem \ref{thm-main}]
The Bishop--Gromov inequality gives, for every $R>0$,
\begin{equation}\label{eq:MassBoundBall}
    \sup_i\|T_i\|(B_R(x_i))<\infty.
\end{equation}
Since $\partial T_i=0$, the hypotheses of
Lemma~\ref{lem:common-pgh-localflat} are satisfied.

After passing to a subsequence, we therefore obtain a complete metric
space $Z$, isometric embeddings
\[
    \iota_i:X_i\to Z,
    \qquad
    \iota:X\to Z,
\]
and a current
$\widetilde T\in\LWic{Z}{N}$
such that
\begin{equation}\label{eq:pointed-common-localflat}
\iota_{i\sharp}T_i \to
    \widetilde T
\quad\text{locally flat in }Z, \qquad \iota_i(x_i)\to\iota(x),
\end{equation}
and
\begin{equation}\label{eq:pointed-support-in-X}
\spt(\widetilde T)\subset\iota(X).
\end{equation}

If $\haus^N(B_1(x_i))\to0$,
then Bishop--Gromov implies
$\haus^N(B_R(x_i))\to0$
for every fixed $R>0$. Since
\(
    \|T_i\|=\haus^N,
\)
local lower semicontinuity of the mass in
\eqref{eq:pointed-common-localflat} gives
\(
    \widetilde T=0.
\)
This proves the collapsing alternative.

Assume from now on that the sequence is non-collapsing. Volume
convergence gives
\[
    (X_i,\dist_i,\haus^N,x_i)
    \to
    (X,\dist,\haus^N,x)
\]
in the pointed measured Gromov--Hausdorff topology. In particular,
$(X,\dist,\haus^N)$ is a proper non-collapsed $\RCD(K,N)$ space.
Moreover, since the $X_i$ have no boundary, the limit $X$ has no
boundary by \cite[Theorem~5.1]{KapMonGT}.

By \eqref{eq:pointed-support-in-X}, we can regard $\widetilde T$ as a
current on $\iota(X)$. Define
\[
    T:=(\iota^{-1})_\sharp\widetilde T
    \in\LWic{X}{N}.
\]
Since local flat convergence commutes with the boundary and
$\partial T_i=0$, we obtain $\partial T=0$.

We now prove that $\|T\|=\haus^N$.
Fix an arbitrary regular point
$p\in\mathcal R\subset X$.
Since $\iota_i(X_i)$ converges locally in Hausdorff distance to
$\iota(X)$, we can choose points $p_i\in X_i$ such that
\begin{equation}\label{eq:pointed-pi-to-p-common}
    \iota_i(p_i)\to\iota(p)
    \qquad\text{in }Z.
\end{equation}

Choose the splitting maps $u_i$ and their limit $u$ compatibly, exactly
as in the proof of Theorem~\ref{thm-mainD}; in particular, after passing
to a further subsequence, $u_i\to u$ 
locally uniformly in the pmGH sense, and the canonical McShane
extensions of their components to $Z$ satisfy the uniform convergence
required in Lemma~\ref{le:L1DegConv}.

Next choose
\[
    0<r<R<R_0
\]
as in the proof of Theorem~\ref{thm-mainD}. Applying the local-flat
version of Lemma~\ref{lem:localization-flat-balls} to
\[
    S_i:=\iota_{i\sharp}T_i,
    \qquad
    S:=\widetilde T,
    \qquad
    z_i:=\iota_i(p_i),
    \qquad
    z:=\iota(p),
\]
and using \eqref{eq:pointed-pi-to-p-common}, we may choose $R$ so that
\begin{equation}\label{eq:pointed-localized-flat}
    (\iota_{i\sharp}T_i)\rstr B_R(\iota_i(p_i))
    \to
    \widetilde T\rstr B_R(\iota(p))
\end{equation}
in flat distance. The remainder of the argument is local and is identical
to the corresponding part of the proof of
Theorem~\ref{thm-mainD}.
For each fixed $\varepsilon>0$ and every sufficiently
small $r$, take $R_0=2r$ and choose the slicing radius
$R\in(r,3r/2)$ as above.
The compatible coordinate extensions, degree convergence,
and localization to the common Euclidean ball give
\eqref{eq:lower-density-new} for the same fixed current $T$.
Thus
\[
 \liminf_{r\downarrow0}
 \frac{\|T\|(B_r(p))}{\omega_Nr^N}\geq1
\]
at every regular point $p$.
Volume convergence and lower semicontinuity give
$\|T\|\leq\haus^N$, so the reverse density inequality
holds as well.
Differentiation of measures therefore yields
$\|T\|=\haus^N$.

Finally, Proposition~\ref{prop-ex-assms} and
Lemma~\ref{lem:set} imply $\set(T)=X$.
Thus $T$ is an orientation of $(X,\dist,\haus^N)$. In particular, in the common realization $Z$, $\spt(\widetilde T)= \iota(X)$,
so the pointed Gromov--Hausdorff and local flat limits have the same
underlying space.
\end{proof}

We conclude by addressing the analogous results in the Kato class. To this aim, we first establish the next lemma.

\begin{lemma}[Regular balls in the Kato setting]
\label{lem:kato-regular-balls}
Let $(M_i,g_i,x_i)$ satisfy the assumptions of
Theorem~\ref{thm-main2}, and write $\dist_i=\dist_{g_i}$.
Then
\[
 (M_i,\dist_i,\vol_{g_i},x_i)
 \to
 (X,\dist,\haus^N,x)
\]
in the pointed measured Gromov--Hausdorff topology.
There is a set $\mathcal R\subset X$ of full $\haus^N$-measure
such that
\[
 \lim_{s\downarrow0}
 \frac{\haus^N(B_s(p))}{\omega_Ns^N}=1
 \qquad\text{for every }p\in\mathcal R.
\]

Fix $p\in\mathcal R$, points $p_i\to p$, and
$\alpha\in(1/2,1)$. There exists $r_\alpha(p)>0$ such that,
for every $0<r<r_\alpha(p)$ and all sufficiently large $i$,
there is a harmonic map
\[
 u_i:B_{6r}(p_i)\to\mathbb R^N,
 \qquad u_i(p_i)=0,
\]
whose restriction to $B_{3r}(p_i)$ is a diffeomorphism onto
its image and satisfies
\begin{equation}\label{eq:kato-chart-bounds}
 \alpha(6r)^{1-1/\alpha}\dist_i(y,z)^{1/\alpha}
 \leq |u_i(y)-u_i(z)|
 \leq \alpha^{-1}\dist_i(y,z)
\end{equation}
for all $y,z\in B_{3r}(p_i)$, together with
\[
 B^{\mathbb R^N}_{3\alpha r}(0)
 \subset u_i(B_{3r}(p_i)).
\]

For each fixed $\alpha$ and $r$, after passing to a
subsequence, these maps converge locally uniformly in
the Gromov--Hausdorff sense on $B_{3r}(p)$ to an injective
map
\[
 u:B_{3r}(p)\to\mathbb R^N,
 \qquad u(p)=0,
\]
which satisfies the same two-sided distance estimate.
Explicitly, if $q_i\to q\in B_{3r}(p)$, then
$u_i(q_i)\to u(q)$.

Set
\[
 a_\alpha:=\alpha\,6^{1-1/\alpha},
 \qquad
 E_{\alpha,r}:=B^{\mathbb R^N}_{a_\alpha r}(0).
\]
Then
\begin{equation}\label{eq:kato-common-image}
 E_{\alpha,r}
 \subset u_i(B_r(p_i))\cap u(B_r(p))
\end{equation}
for all sufficiently large $i$. Moreover,
\begin{equation}\label{eq:kato-localization}
 u^{-1}(E_{\alpha,r})\cap B_{3r}(p)\subset B_r(p).
\end{equation}
The analogous inclusion holds for $u_i$.
Finally, $X$ satisfies {\rm(LDB)}.
\end{lemma}

\begin{proof}
Volume convergence and the positive lower density bound
follow from
\cite[Proposition~5.9 and its proof]
{CarronMondelloTewodrose2023}.
The full-measure regular set and its unit volume density
are supplied by the non-collapsed Kato regularity theory
\cite{CMT:AFM}.
Here and below, we use its complete-manifold version
under the integrability condition
$\int_0^{t_0}f(t)\,dt/t<\infty$; see
\cite[Section~4.2, Remark~4.5 and Section~5.3]
{CarronMondelloTewodrose2023}.

Fix $\alpha\in(1/2,1)$.
Let $\delta>0$ be the threshold in
\cite[Theorem~5.11]{CarronMondelloTewodrose2023},
and let $\eta>0$ be the corresponding threshold in the
last assertion of
\cite[Remark~5.12]{CarronMondelloTewodrose2023},
decreasing $\eta$ so that $\eta<\delta$.
Since $p$ has unit volume density and $f(t)\to0$,
we may choose $r_\alpha(p)>0$ so that, for every
$0<r<r_\alpha(p)$,
\[
 6r<\sqrt{t_0},\qquad
 f((6r)^2)<\eta,\qquad
 \haus^N(B_{12r}(p))
 >(1-\eta/2)\omega_N(12r)^N.
\]
By measured convergence and the lower semicontinuity
of the measure of open balls, for all sufficiently
large $i$,
\[
 \vol_{g_i}(B_{12r}(p_i))
 >(1-\eta)\omega_N(12r)^N.
\]
\cite[Remark~5.12]{CarronMondelloTewodrose2023} therefore supplies a harmonic map on
$B_{6r}(p_i)$ satisfying the smallness hypothesis of
\cite[Theorem~5.11]{CarronMondelloTewodrose2023}. Such result gives \eqref{eq:kato-chart-bounds},
the diffeomorphism assertion, and the stated image
inclusion.

The uniform Lipschitz bound and the normalization at
$p_i$ give, by Arzel\'a--Ascoli on compact subballs and
a diagonal argument, a limit map $u$ on $B_{3r}(p)$.
Both distance estimates pass to the limit; in
particular, $u$ is injective.

If $v\in E_{\alpha,r}$, the image inclusion on
$B_{3r}(p_i)$ gives a point $q_i$ with $u_i(q_i)=v$.
The radial lower bound gives
\[
 \alpha(6r)^{1-1/\alpha}\dist_i(p_i,q_i)^{1/\alpha}
 \leq |v|<a_\alpha r,
\]
and hence $\dist_i(p_i,q_i)<r$.
For this fixed $v$, the same inequality bounds these
distances by a constant strictly smaller than $r$.
A convergent subsequence of the points $q_i$ therefore
gives $q\in B_r(p)$ with $u(q)=v$.
This proves \eqref{eq:kato-common-image}.
The radial lower bound also proves
\eqref{eq:kato-localization} directly.
\end{proof}

\begin{proof}[Proof of Theorem \ref{thm-main2}]
Let $T_i$ be the current associated with the given
orientation of $M_i$. Then
\[
 T_i\in\LWic{M_i}{N},\qquad
 \|T_i\|=\vol_{g_i},\qquad
 \partial T_i=0.
\]
By Lemma~\ref{lem:kato-regular-balls}, the Riemannian
volume measures converge to $\haus^N$.
In particular, for every $R>0$,
\[
 \sup_i\vol_{g_i}(B_R(x_i))<\infty.
\]
Indeed, one may choose $R'>R$ with
$\haus^N(\partial B_{R'}(x))=0$ and use volume
convergence on $B_{R'}(x_i)$; the finitely many
remaining indices cause no difficulty.

Lemma~\ref{lem:common-pgh-localflat} therefore gives,
after passing to a subsequence, a common complete
metric space $Z$, isometric embeddings $\iota_i,\iota$,
and a locally integral current $\widetilde T$ such that
\[
 \iota_{i\sharp}T_i\to\widetilde T
 \quad\text{locally flat in }Z,\qquad
 \spt(\widetilde T)\subset\iota(X).
\]
Set
\[
 T:=(\iota^{-1})_\sharp\widetilde T
 \in\LWic{X}{N}.
\]
Then $\partial T=0$.
Lower semicontinuity of mass, together with convergence
of the volume measures in the common realization,
also gives the measure inequality
\begin{equation}\label{eq:kato-mass-domination}
 \|T\|\leq\haus^N.
\end{equation}

Fix $p\in\mathcal R$ and $\alpha\in(1/2,1)$.
For an arbitrary $0<r<r_\alpha(p)$, choose the maps
$u_i,u$ given by Lemma~\ref{lem:kato-regular-balls}.
All further subsequences in the argument below retain
the already fixed current limit $T$.

By localization of local flat convergence and slicing,
after passing to a further subsequence we can choose
$R\in(r,2r)$ such that
\[
 S_i:=\iota_{i\sharp}(T_i\llcorner B_R(p_i))
 \to
 S:=\iota_\sharp(T\llcorner B_R(p))
\]
in flat distance, with
\[
 \sup_i\bigl(\mathbf M(S_i)+
                 \mathbf M(\partial S_i)\bigr)<\infty.
\]
We also choose $R$ so that
$\haus^N(\partial B_R(p))=0$.

On the closed $R$-balls, the maps $u_i$ converge
compatibly to $u$ and have a uniform Lipschitz bound.
Taking the componentwise canonical McShane extensions
to $Z$, as in
\eqref{eq:canonical-McShane-ui}--%
\eqref{eq:McShane-uniform-estimate},
gives maps $\widehat u_i,\widehat u:Z\to\mathbb R^N$
with uniformly Lipschitz components and
\[
 \sum_{\beta=1}^N
 \|\widehat u_i^\beta-\widehat u^\beta\|_{L^\infty(Z)}
 \to0.
\]
Lemma~\ref{le:L1DegConv} consequently yields
\begin{equation}\label{eq:kato-degree-convergence}
 \deg(S_i,\widehat u_i,\cdot)
 \to
 \deg(S,\widehat u,\cdot)
 \quad\text{in }L^1(\mathbb R^N).
\end{equation}

Write $E=E_{\alpha,r}$.
The map $u_i$ is a diffeomorphism on the connected
ball $B_{3r}(p_i)$, so its orientation sign is a
constant $\sigma_i\in\{-1,1\}$.
Every point of $E$ has exactly one preimage in
$B_R(p_i)$, and that preimage belongs to $B_r(p_i)$.
The change-of-variables formula for the smooth
orientation current therefore gives
\[
 \deg(S_i,\widehat u_i,v)=\sigma_i
 \quad\text{for }\mathcal L^N\text{-a.e. }v\in E.
\]
Since $\mathcal L^N(E)>0$, the $L^1$ convergence in
\eqref{eq:kato-degree-convergence} implies that
$\sigma_i$ is eventually constant.
Thus, for some $\sigma\in\{-1,1\}$,
\[
 \deg(S,\widehat u,v)=\sigma
 \quad\text{for }\mathcal L^N\text{-a.e. }v\in E.
\]

Set
\[
 A:=u^{-1}(E)\cap B_R(p).
\]
By Lemma~\ref{lem:kato-regular-balls}, $A\subset B_r(p)$.
Restriction of the pushforward to $E$ therefore gives
\[
 u_\sharp(T\llcorner A)=\sigma[[E]].
\]

Using the Lipschitz constant $\alpha^{-1}$ of $u$
on $B_{3r}(p)$, we obtain
\[
 \omega_N(a_\alpha r)^N
 =\mathcal L^N(E)
 \leq \alpha^{-N}\|T\|(A)
 \leq \alpha^{-N}\|T\|(B_r(p)).
\]
Consequently, for every $0<r<r_\alpha(p)$,
\[
 \frac{\|T\|(B_r(p))}{\omega_Nr^N}
 \geq (\alpha a_\alpha)^N.
\]
Since $a_\alpha=\alpha6^{1-1/\alpha}\to1$ as
$\alpha\uparrow1$, first letting $r\downarrow0$
and then $\alpha\uparrow1$ gives
\[
 \liminf_{r\downarrow0}
 \frac{\|T\|(B_r(p))}{\omega_Nr^N}\geq1.
\]
The reverse inequality follows from
\eqref{eq:kato-mass-domination} and the unit volume
density at $p$. Hence
\[
 \lim_{r\downarrow0}
 \frac{\|T\|(B_r(p))}{\omega_Nr^N}=1
 \qquad\text{for every }p\in\mathcal R.
\]

Since $\mathcal R$ has full $\haus^N$-measure,
differentiation of measures gives
$\|T\|=\haus^N$.
Finally, {\rm(LDB)} follows from
Lemma~\ref{lem:kato-regular-balls}, so
Lemma~\ref{lem:set} implies $\set(T)=X$.
Thus $T$ is an orientation of
$(X,\dist,\haus^N)$ and $\partial T=0$.
The common realization above supplies the asserted
local flat convergence.
\end{proof}

\bibliographystyle{plain}
\bibliography{Biblio}

\end{document}